\documentclass[11pt]{article}
\usepackage[utf8]{inputenc}

\usepackage{geometry} % to change the page dimensions
\usepackage{graphicx} % support the \includegraphics command and options

\usepackage{tabularx}
\newcolumntype{Z}{>{\centering\let\newline\\\arraybackslash\hspace{0pt}}X}
\usepackage{ulem}
\usepackage{booktabs} % for much better looking tables
\usepackage{array} % for better arrays (eg matrices) in maths
\usepackage{paralist} % very flexible & customisable lists (eg. enumerate/itemize, etc.)
\usepackage{verbatim} % adds environment for commenting out blocks of text & for better verbatim
\usepackage{subfig} % make it possible to include more than one captioned figure/table in a single float
\usepackage{hyperref}
\usepackage{amsmath}
\usepackage{amssymb}
\usepackage{authblk}
\usepackage{stmaryrd}
\usepackage{enumitem}
\setlist{nosep}
\usepackage{bbm}
\usepackage{mathtools}
\usepackage{textcomp}
\usepackage{xcolor}
\usepackage{listings}
\usepackage{float}
\usepackage{centernot}
\usepackage[mathscr]{euscript}
\usepackage{tikz-cd}
\usepackage{rotating}

\newcommand{\Ccal}{\mathcal{C}}
\newcommand{\Dcal}{\mathcal{D}}
\newcommand{\Ecal}{\mathcal{E}}
\newcommand{\Fcal}{\mathcal{F}}

\newcommand{\Ocal}{\mathcal{O}}

\newcommand{\Nbb}{\mathbb{N}}

\newcommand{\Set}{\mathbf{Set}}
\newcommand{\Setswith}[1]{\PSh(#1)}

\newcommand{\Hom}{\mathrm{Hom}}
\newcommand{\HOM}{\mathcal{H}\! \mathit{om}}

\newcommand{\op}{^{\mathrm{op}}}

\newcommand{\too}{\twoheadrightarrow}

\DeclareMathOperator{\id}{id}

\DeclareMathOperator{\Fix}{Fix}
\DeclareMathOperator{\PSh}{\mathbf{PSh}}
\DeclareMathOperator{\Sh}{\mathbf{Sh}}

\DeclareMathOperator{\End}{End}

\tikzset{
  no line/.style={draw=none,
    commutative diagrams/every label/.append style={/tikz/auto=false}},
  from/.style args={#1 to #2}{to path={(#1)--(#2)\tikztonodes}}
	}
\tikzset{symbol/.style={draw=none, every to/.append style={edge node = {node [sloped, allow upside down, auto=false] {$#1$}}}}}

\usepackage{amsthm}
\newtheorem{thm}{Theorem}[section]
\newtheorem{theorem}[thm]{Theorem}
\newtheorem{prop}[thm]{Proposition}
\newtheorem{proposition}[thm]{Proposition}
\newtheorem{lemma}[thm]{Lemma}
\newtheorem{crly}[thm]{Corollary}
\newtheorem{corollary}[thm]{Corollary}

\theoremstyle{definition}
\newtheorem{dfn}[thm]{Definition}
\newtheorem{definition}[thm]{Definition}
\newtheorem{xmpl}[thm]{Example}
\newtheorem{example}[thm]{Example}

\theoremstyle{remark}
\newtheorem{rmk}[thm]{Remark}
\newtheorem{remark}[thm]{Remark}

\newtheorem{fact}[thm]{Fact}

\makeatletter
\usepackage{tikz}
\newcommand*\circled[2][1.6]{\tikz[baseline=(char.base)]{
    \node[shape=circle, draw, inner sep=1pt, 
        minimum height={\f@size*#1},] (char) {\vphantom{WAH1g}#2};}}
\makeatother

\usepackage{fancyhdr} % This should be set AFTER setting up the page geometry
\usepackage{sectsty}
\allsectionsfont{\sffamily\mdseries\upshape} % (See the fntguide.pdf for font help)
\usepackage[nottoc,notlof,notlot]{tocbibind} % Put the bibliography in the ToC
\usepackage[titles,subfigure]{tocloft} % Alter the style of the Table of Contents

\usepackage[nodayofweek]{datetime}
\longdate

\let\theta\vartheta
\let\emph\textit

\date{} % Activate to display a given date or no date (if empty),
\AtBeginDocument{%
   \def\MR#1{}
}

\title{Monoid Properties as Invariants of Toposes of Monoid Actions}
\author{Jens Hemelaer \thanks{Department of Mathematics, University of Antwerp, 
 Middelheimlaan 1, B-2020 Antwerp (Belgium) \\ email: jens.hemelaer@uantwerpen.be} \\ Morgan Rogers \thanks{Universit\`a degli Studi dell{'}Insubria, Via Valleggio n. 11, 22100 Como CO \\ Marie Sklodowska-Curie fellow of the Istituto Nazionale di Alta Matematica \\ email: mrogers@uninsubria.it}}

\begin{document}

\maketitle

\begin{abstract}
We systematically investigate, for a monoid $M$, how topos-theoretic properties of $\Setswith{M}$, including the properties of being atomic, strongly compact, local, totally connected or cohesive, correspond to semigroup-theoretic properties of $M$.
\end{abstract}

\tableofcontents

\section*{Introduction}

Let $M$ be a monoid, viewed as a category with one object. The category of right $M$-sets, also called $M$-acts in the semigroup theory literature, is precisely the presheaf topos $\Setswith{M}$. Recently, the authors of the present work established representation theorems for toposes of the form $\Setswith{M}$; that is, toposes of right $M$-sets for a discrete monoid $M$. In \cite{TDMA}, the second named author showed using topos-theoretic methods that any topos with a surjective essential point (equivalently, admitting a monadic and comonadic functor to the topos $\Set$ of sets) is of this form. The first named author showed in \cite{TGRM} that such a topos can also be represented as the topos of equivariant sheaves on a posetal groupoid.

The topos $\Setswith{M}$ holds a great deal more structure than the monoid $M$ alone. In particular, it is the natural setting in which to define a great variety of constructions and tools for examining the subtler properties of monoids. Perhaps more significantly, being a Grothendieck topos, the category $\Setswith{M}$ can be compared, either indirectly through its properties or directly via equivalences or geometric morphisms, to other toposes. Studying these comparisons between toposes in order to get a better understanding of the structures presenting them (in this case discrete monoids) forms the backbone of the `toposes as bridges' philosophy of Caramello \cite{TST}. In this paper, we take the indirect approach, investigating correspondences between properties of the representing monoids and well-understood properties of the corresponding toposes from the topos-theoretic literature. In particular, it was noted in \cite{TDMA} that several important classes of monoid, the representing monoid is unique up to isomorphism. It follows that the properties identifying these classes are Morita-invariant, and should therefore correspond to topos-theoretic invariants, some of which we have been able to identify.

The ``purely semigroup-theoretic content'' of many of the results presented in this article have turned out to be known results, in that after deriving them we discovered references for them in existing literature. However, this is typical when establishing a category-theoretic approach to any area of mathematics: reproving elementary results in context is a necessary first step in applying topos-theoretic machinery, since it illuminates the efficacy and potential for generalisation of this approach. We believe that topos theory will ultimately be a fruitful source of new results in semigroup theory. Reciprocally, monoids shall provide a useful source of examples, properties, constructions and intuition for toposes distinct from the usual geometric and logical perspectives, and simpler than the larger context of presheaf toposes over categories with several objects. We hope this article will achieve our goal of drawing together the research communities in semigroup theory and topos theory.

There has been recent interest in the toposes $\Setswith{M}$ from a geometrical point of view. Connes and Consani, in their construction of the Arithmetic Site \cite{connes-consani} \cite{connes-consani-geometry-as}, considered the special case where $M$ is the monoid of nonzero natural numbers under multiplication. In this case, the points of $\Setswith{M}$ are related to the finite adeles in number theory. Related toposes are studied in \cite{sagnier}, \cite{arithmtop} and \cite{llb-three}. As mentioned in \cite{TGRM}, this geometric study of monoids is inspired by the idea of ``algebraic geometry over $\mathbb{F}_1$'' \cite{manin}. In this philosophy, commutative monoids are thought of as dual to affine $\mathbb{F}_1$-schemes, while the topos $\Setswith{M}$ is seen as the topos of quasi-coherent sheaves on the space corresponding to the monoid $M$, see \cite{pirashvili}.

This paper is organised as follows. In Section \ref{sec:bg} we recap the necessary topos theory and the relevant content of the papers referenced above, including the functors constituting the global sections geometric morphism which is intrinsic to a Grothendieck topos. This geometric morphism automatically has many restrictive properties in the case of $\Setswith{M}$ for any monoid $M$, which we establish. We also present a tensor-hom description of adjunctions applicable to the components of the geometric morphisms.

In the main body of the paper, Section \ref{sec:1down}, we examine the consequences of additional properties of $\Setswith{M}$ on the monoid $M$. A great number of properties of geometric morphisms in the topos theory literature are inspired from geometry, since any topological space or locale has an associated Grothendieck topos (its category of sheaves), and the global sections functor of such a topos has properties determined by the properties of the original space; we therefore accompany our exposition for monoids with the parallel analysis for topological spaces, for which results are readily available in the topos theory literature. Our arguments are guided in some cases by the examples of toposes of group actions periodically presented in Johnstone's reference text \cite{Ele}.

In the Conclusion we explain how future work will extend the results accumulated in this paper. We are aware of some properties of monoids which are expressible in terms of categorical properties of $\Setswith{M}$ but which we have not (yet) been able to express in terms of the global sections morphism; we outline these and some further properties of toposes which we intend to explore. In \cite{TDMA}, beyond the characterisation of toposes of the form $\Setswith{M}$, the stronger result of a 2-equivalence between a 2-category of monoids and a 2-category on the corresponding toposes was demonstrated. We explain how another direction of research will take advantage of this 2-equivalence, and how the results obtained in this paper might be relativised. 

\subsection*{Acknowledgements}

The second named author was supported in this work by INdAM and the Marie Sklodowska-Curie Actions as a part of the \textit{INdAM Doctoral Programme in Mathematics and/or Applications Cofunded by Marie Sklodowska-Curie Actions}. We would like to thank the organisers of the sixth Symposium on Compositional Structures conference (SYCO6) in Leicester, UK in 2019 and the Topics in Category Theory School in Edinburgh, UK in 2020 for giving us the opportunity to present early versions of this work. We would also like to thank Olivia Caramello for facilitating our collaboration on this work and for her useful insights, and to appreciate the enthusiastic correspondence from Steve Vickers following the presentation at SYCO6.

\section{Background}
\label{sec:bg}

\subsection{Toposes of Discrete Monoid Actions}
\label{ssec:toposes}

For a reader with a limited background in topos theory, a (by now classical) introductory text is Mac Lane and Moerdijk's \cite{MLM}. This in particular includes exercises (at the end of Chapter I, for example) for identifying topos-theoretic structures in toposes of the form $\Setswith{M}$ and more generally in presheaf toposes; we recall some of this structure here. For two right $M$-sets $X$ and $Y$, we will use the shorthand
\begin{equation*}
\HOM_M(X,Y) = \Hom_{\Setswith{M}}(X,Y)
\end{equation*}
to denote the set of morphisms from $X$ to $Y$.

A topos of the form $\Setswith{M}$ comes equipped with two geometric morphisms: a canonical point and its global sections morphism.
\begin{equation} \label{eq:can-point-and-global-sections}
\begin{tikzcd}
\Set \ar[r, bend left=55, "- \times M"] \ar[r, bend right=45, "{\Hom_{\Set}(M{,}-)}"',pos=.55]
\ar[r, phantom, shift left=6, "\bot", near end] \ar[r, phantom, shift right=4, "\bot", near end] &
{\Setswith{M}} \ar[l, "U"', near start] \ar[r, "C", bend left = 50] \ar[r, "\Gamma"', bend right= 40]
\ar[r, symbol = \bot, shift right = 4, near start] \ar[r, symbol = \bot, shift left = 6, near start] &
\Set \ar[l, "\Delta"', near end],
\end{tikzcd}
\end{equation}
where the functors not explicitly specified are:
\begin{itemize}
	\item The forgetful functor $U$ sending a right $M$-set to its underlying set,
	\item The global sections functor $\Gamma$ sending an $M$-set $X$ to its set
\begin{equation*}
	\Fix_M(X)= \HOM_M(1,X)
\end{equation*}
	of fixed points under the action of $M$,
	\item The constant sheaf functor $\Delta$ sending a set $Y$ to the same set with trivial $M$-action,
	\item The connected components functor $C$ sending an $M$-set $X$ to its set of components under the action of $M$ (that is, to its quotient under the equivalence relation generated by $x \sim x \cdot m$ for $x \in X$, $m \in M$).
\end{itemize}
It should also be noted that for a set $Y$, the $M$-action on $\Hom_{\Set}(M{,}Y)$ by $m \in M$ sends $f$ to $(n \mapsto f(mn))$, while the $M$ action on $Y \times M$ is by right multiplication on the $M$-component.

Recall that in general a \textbf{point} of a topos $\Ecal$ is a geometric morphism $\Set \to \Ecal$; a topos is said to have \textbf{enough points} if the inverse images functors of all its points are jointly conservative. The inverse image functor of the canonical point of $\Setswith{M}$ is the forgetful functor, which is thus automatically conservative, so every topos of this form has enough points.

While the canonical point is determined by the choice of the presenting monoid $M$, the global sections geometric morphism is uniquely determined by the topos, so that its properties are automatically Morita-invariant (that is, independent of the choice of presenting monoid). In this paper we shall primarily concern ourselves with analysis of the global sections morphism. Since the inverse image functor $\Delta$ of this morphism has an (automatically $\Set$-indexed) left adjoint $C$, the global sections morphism of $\Setswith{M}$ meets the definition of being \textbf{locally connected} found at the beginning of \cite[Section C3.3]{Ele}. We equivalently say that $\Setswith{M}$ is locally connected for any $M$. 

Toposes are cartesian closed, which is to say that for each pair of objects $P,Q$ in a topos $\Ecal$ there is an \textbf{exponential object} $Q^P$ such that for any third object $X$, we have an isomorphism 
\begin{equation} \label{eq:exponential-adjunction}
\Hom_{\Ecal}(X\times P,Q) \cong \Hom_{\Ecal}(X,Q^P) 
\end{equation}
natural in $X$ and $Q$, i.e.\ the functor $(-)^P$ is right adjoint to $- \times P$. In particular, for $X = Q^P$, the identity map on the right hand side corresponds on the left hand side to the \textbf{evaluation map} 
\begin{equation} \label{eq:evaluation-map}
\mathrm{ev}: Q^P \times P \longrightarrow Q.
\end{equation}

In $\Ecal = \Setswith{M}$, for two $M$-sets $P,Q$, the exponential $Q^P$ has as underlying set $\HOM_M(M\times P,Q)$. The right $M$-action is defined as
\begin{equation} \label{eq:exp1}
(f \cdot m)(n,p) = f(mn,p),
\end{equation} 
for $f \in Q^P$, $m,n \in M$, $p \in P$ (see \cite[I.6]{MLM}). The evaluation map is then given by
\begin{align} \label{eq:exp2}
\begin{split}
\HOM_M(M\times P,Q) \times P \longrightarrow Q \\
(f,p) \mapsto f(1,p)
\end{split}
\end{align}

If $F:\Fcal \to \Ecal$ is a functor preserving binary products, then there is a natural comparison morphism 
\begin{equation}\label{eq:theta}
\theta_{P,Q}: F(Q^P) \to F(Q)^{F(P)}
\end{equation}
obtained by applying $F$ to the evaluation map (\ref{eq:evaluation-map}), and then transposing back across the product-exponential adjunction in $\Ecal$. If $\theta_{P,Q}$ is a isomorphism for every pair of objects $P,Q$, then we say that $F$ is \textbf{cartesian-closed}, or that $F$ \textbf{preserves exponentials}. Note that the inverse image functor of a locally connected geometric morphism is always cartesian closed, by \cite[Proposition C3.3.1]{Ele}, so that in particular $\Delta$ is cartesian closed. The condition of cartesian-closedness can be weakened in two directions: either by restricting to a smaller collection of pairs $P,Q$ of objects on which $\theta_{P,Q}$ is required to be an isomorphism, or by asking that $\theta_{P,Q}$ have a property weaker than being an isomorphism. We discuss some cases of the former, applied to the connected components functor $C$, in Section \ref{ssec:exponential} and cases of the latter in Section \ref{ssec:trivial}.

Toposes also have subobject classifiers. That is, there is an object $\Omega_{\Ecal}$ in a topos $\Ecal$ equipped with a subobject $\top:1 \hookrightarrow \Omega_{\Ecal}$ such that every subobject $S \hookrightarrow A$, for $A$ an object of the topos, is the pullback of $\top$ along a unique classifying morphism $s: A \to \Omega_{\Ecal}$.
\begin{equation}
\begin{tikzcd}
S \ar[r] \ar[dr, phantom, "\lrcorner", very near start] \ar[d, hook] & 1 \ar[d,"{\top}", hook] \\
A \ar[r, "s"'] & \Omega_{\mathcal{E}}
\end{tikzcd}
\end{equation}

The subobject classifier of $\Ecal = \Setswith{M}$ is the set $\Omega$ of right ideals of $M$ equipped with the inverse image action for left multiplication, which for a right ideal $I$ and $m \in M$ is defined by $I\cdot m = \{m'\in M \mid mm' \in I\}$. The morphism $\top: 1 \to \Omega$ identifies $M \in \Omega$ as the largest ideal of $M$, so that the subobject $S$ of an object $A$ classified by a morphism $s: A \to \Omega$ is the subset of $A$ on the elements $a$ with $s(a) = M$.

If $F:\Fcal \to \Ecal$ is a functor preserving monomorphisms, then there is a canonical morphism 
\begin{equation} \label{eq:chi}
\chi:F(\Omega_{\Fcal}) \to \Omega_{\Ecal}
\end{equation}
classifying the subobject $F(\top)$ of $F(\Omega_{\Fcal})$. $F$ is said to \textbf{preserve the subobject classifier} if $\chi$ is an isomorphism. The direct image functor of a geometric morphism preserves the subobject classifier if and only if it is \textbf{hyperconnected}, by \cite[Proposition A4.6.6(v)]{Ele}. For the global sections morphism of a Grothendieck topos, this requires the subobject classifier to admit exactly two morphisms from $1$, which can be restated as the property that the topos is \textbf{two-valued}, having exactly two subterminal objects. This is true of $\Setswith{M}$: its terminal object is the one-element $M$-set $1$, whose only proper sub-$M$-set is empty. Thus $\Gamma$ preserves the subobject classifier, and $\Setswith{M}$ is hyperconnected over $\Set$, for any monoid $M$.

Finally, if a functor $F$ preserves monomorphisms and finite products, we say $F$ is \textbf{logical} if it preserves the subobject classifier and exponentials. A geometric morphism whose direct image is logical is automatically an equivalence (\cite[A4.6.7]{Ele}), so $\Gamma$ is logical if and only if $M$ is trivial; this appears as a condition in Theorem \ref{thm:trivial}. We shall see what happens when $\Delta$ is logical in Theorem \ref{thm:atomic}. However, since $\Delta$ already has so many strong preservation properties for an arbitrary monoid $M$, always having a left and right adjoint and preserving exponentials, that theorem is the only one we identify in this paper expressed in terms of properties of $\Delta$.

Several of the properties of toposes we examine are geometric, in the sense that they inherit their names from properties of toposes $\Sh(X)$ of sheaves on a topological space (or more generally a locale), $X$. Accordingly, we supplement most of the definitions in this paper with an illustration of what they mean for toposes of this form. In doing so, it will occasionally benefit us to exploit the equivalence, demonstrated in \cite[Corollary II.6.3]{MLM}, between $\Sh(X)$ and the category $\mathbf{LH}/X$ of \textbf{local homeomorphisms over $X$}, whose objects are \textit{local homeomorphisms} or \textit{\'etale maps} $E \to X$, and whose morphisms are continuous maps making the resulting triangle over $X$ commute. Passing through this equivalence, the components of the global sections morphism $f:\Sh(X) \to \Set$ acquire new interpretations. $f_*(\pi: E \to X)$ is the set of global sections of $\pi$, i.e.\ the set of continuous maps $s : X \to E$ such that $\pi \circ s = \mathrm{id}_X$ (this is where the name of this functor comes from for a general topos). $f^*$ sends a set $A$ to the $A$-fold cover $\pi_1:X \times A \to X$ of $X$. The inverse image functor $f^*$ has a left adjoint $f_!$ if and only if $X$ is a locally connected space, in which case $f_!(\pi:E \to X)$ is the set of connected components of $E$, which is where the name of locally connected geometric morphisms originates. Finally, $X$ is connected if and only if $f^*$ is full and faithful, which is why a geometric morphism with this property is called connected. Since hyperconnectedness of a topos implies connectedness, toposes of monoid actions are connected and locally connected over $\Set$, whence toposes of sheaves over connected, locally connected spaces are a good source of intuition for their properties. It should be stressed, however, that the global sections morphism of $\Sh(X)$ is only hyperconnected if $X$ is the one-point space, so these comparisons can never be realised as equivalences of toposes outside of the case where both the space and monoid are trivial. This is a strength of our property-oriented approach: it allows us to draw formal comparisons between classes of objects even when their corresponding toposes do not coincide.

As a final note, in our examples we will always talk about the topos $\Sh(X)$ of sheaves on a \textit{sober topological space} $X$, which means that the points of $X$ correspond bijectively with the points of $\Sh(X)$, and that this topos has enough points. In instances where the requirement of having enough points is not explicitly mentioned, the results can be extended to encompass toposes of sheaves on suitable locales, should the reader desire it.

\subsection{Tensors and Homs, Flatness and Projectivity}

Being a Grothendieck topos, the category of right $M$-sets has all limits and colimits. Since the functor $U$ from \eqref{eq:can-point-and-global-sections} preserves both limits and colimits, it follows that colimits and limits can be computed on underlying sets. We will use the following notations:
\begin{itemize}
\item $\varnothing$ for the initial object, i.e.\ the empty right $M$-set;
\item $1$ for the terminal object, i.e.\ the right $M$-set with one element;
\item $A \sqcup B$ for the coproduct (disjoint union) of two right $M$-sets $A$ and $B$;
\item $\bigsqcup_{i \in I} X_i$ for the coproduct (disjoint union) of a family of right $M$-sets $\{X_i\}_{i \in I}$;
\item $\varinjlim_{i \in I} X_i$ for the colimit of a filtered diagram $\{X_i\}_{i \in I}$.
\end{itemize}

\subsubsection{Hom-sets and projectivity}

For a monoid $N$ and two right $N$-sets $A$ and $B$, we can consider the hom-set
\begin{equation}
\HOM_N(B,A),
\end{equation}
so that for fixed $B$ we get a functor $\Setswith{N} \to \Set$. Clearly, the global sections functor $\Gamma$ of $\Setswith{N}$ can be expressed as $\HOM_N(1,-)$, for $1$ the trivial right $N$-set, so that properties of $\Gamma$ can be expressed as properties of this $N$-set.

Let $M$ be another monoid and suppose that $B$ is additionally equipped with a \textbf{compatible} left $M$-action, such that
\begin{equation} \label{eq:compatible}
(m \cdot b) \cdot n = m \cdot (b\cdot n)
\end{equation} 
for all $m \in M$, $b \in B$ and $n \in N$ (we will say that $B$ is a left-$M$-right-$N$-set). Then the inverse image action on $\HOM_N(B,A)$ makes it a right $M$-set, so that $\HOM_N(B,-)$ defines a functor $\Setswith{N} \to \Setswith{M}$.

The following definitions work in any topos, but for clarity we formulate it in our special case of a topos $\Setswith{N}$ with $N$ a monoid.

\begin{dfn}
\label{dfn:projectives}
Let $B$ be a right $N$-set, and consider the functor $\HOM_N(B,-): A \mapsto \HOM_N(B,A)$. Then we say that $B$ is: 
\begin{itemize}
\item \textbf{connected} or \textbf{indecomposable} if $\HOM_N(B,-)$ preserves arbitrary (small/set-indexed) coproducts;
\item \textbf{projective} if $\HOM_N(B,-)$ preserves epimorphisms;
\item \textbf{finitely presentable} if $\HOM_N(B,-)$ preserves filtered colimits.
\end{itemize}
\end{dfn}

The definitions for left $N$-sets are analogous.

Every right $M$-set can be written as the disjoint union of its connected components. Further:

\begin{proposition} \label{prop:connected-objects}
Let $X$ be a right $M$-set. Then the following are equivalent:
\begin{enumerate}
\item $X$ is connected/indecomposable;
\item $\HOM_M(X,-)$ preserves binary coproducts;
\item $X$ is non-empty, and $X \cong X_1 \sqcup X_2$ implies that either $X_1$ is empty or $X_2$ is empty;
\item $|C(X)|=1$, with $C$ the connected components functor of (\ref{eq:can-point-and-global-sections}).
\end{enumerate}
\end{proposition}
\begin{proof}
This can be shown as an exercise, but it also follows from more general results in topos theory, as follows. The equivalences ($1 \Leftrightarrow 2 \Leftrightarrow 3$) hold in any infinitary extensive category by \cite[Theorem 2.1]{janelidze}. Grothendieck toposes are infinitary extensive, see e.g.\ \cite[\S4.3]{carboni-vitale}. The equivalence $(3 \Leftrightarrow 4)$ holds for locally connected Grothendieck toposes, and is discussed around \cite[Lemma C3.3.6]{Ele}.
\end{proof}

In particular, the empty set is not indecomposable. Note that in the semigroup literature, $M$-sets are sometimes assumed to be non-empty by definition, see e.g.\ \cite{MAC}. We will not follow this convention, because it prevents the category of right $M$-sets from being a topos.

A right $M$-set will be called \textbf{free} if each connected component is isomorphic to $M$ (with right $M$-action given by multiplication). Free and projective right $M$-sets are related in the following way:

\begin{proposition}
For $M$ a monoid and $P$ a right $M$-set, the following are equivalent:
\begin{enumerate}
\item $P$ is projective;
\item $P$ is a retract of a free right $M$-set;
\item $P \cong \bigsqcup_{i \in I} e_i M$ for some family $\{e_i\}_{i \in I}$ of idempotents in $M$. 
\end{enumerate}
\end{proposition}
\begin{proof}
See e.g.\ \cite[III, 17]{MAC}.
\end{proof}

We should note that the topos-theoretic definition of finite presentability agrees with the one from universal algebra:

\begin{proposition}[{Cf. \cite[Corollary 3.13]{adamek-rosicky}}] \label{prop:finitely-presentable-topos-vs-algebra}
A right $M$-set $X$ is finitely presentable if and only if we can write $X$ as the colimit of a diagram
\begin{equation*}
\begin{tikzcd}
F \ar[r,shift left=1,"{a}"] \ar[r,shift right=1,"{b}"'] & F'
\end{tikzcd}
\end{equation*}
for some finitely generated free right $M$-sets $F$ and $F'$ and morphisms $a,b$.
\end{proposition}

\subsubsection{Tensors and flatness}

Consider a right $M$-set $A$ and a left $M$-set $B$. Recall that the \textbf{tensor product} of $A$ and $B$ over $M$ is defined as the set:
\begin{equation*}
A \otimes_M B ~=~ (A \times B)/\!\sim\,
\end{equation*}
where $\sim$ is the equivalence relation generated by
\begin{equation*}
(am,b) \sim (a,mb)
\end{equation*}
for all $a \in A$, $b \in B$ and $m \in M$. The equivalence class of a pair $(a,b)$ is denoted by $a \otimes b$. An alternative expression for the connected components functor is $C(X) = X \otimes_M 1$, where $1$ is the trivial left $M$-set. Thus properties of $C$ can be expressed as properties of this trivial left $M$-set; later in this section we shall establish the terminology we will use for these properties. More generally, suppose that there is a monoid $N$ and a compatible right $N$-action on $B$, as in \eqref{eq:compatible} above. Then $A \otimes B$ inherits a right $N$-action from $B$, defined by $(a \otimes b) \cdot n = a \otimes (b \cdot n)$.

As in ring theory, the functor $-\otimes_M B$ is left adjoint to $\HOM_N(B,-)$, for $B$ a set with compatible left $M$-action and right $N$-action.
\begin{proposition}
\label{prop:adjunction}
Let $M$ and $N$ be monoids. Any left-$M$-right-$N$-set $B$ induces an adjunction:
\[\begin{tikzcd}[column sep=large]
\Setswith{N} \ar[r, bend right = 15, "{\HOM_N(B,-)}"'] \ar[r, phantom, "\bot"] & \Setswith{M} \ar[l, bend right = 15, "- \otimes_M B"']
\end{tikzcd}\]
Conversely, any adjunction between these categories (in this direction) has this form for some left-$M$-right-$N$-set $B$.
\end{proposition}
\begin{proof}
For $X$ a right $M$-set and $Y$ a right $N$-set, the isomorphism $\HOM_M(Y \otimes_M B, X) \cong \HOM_N(Y, \HOM_N(B,X))$ sends $h$ in the former to the mapping $y \mapsto (b \mapsto h(y \otimes b))$. It is easy to check that this is well-defined, and has inverse mapping $k$ to $y \otimes b \mapsto k(y)(b)$.

Conversely, given $F:\Setswith{M} \to \Setswith{N}$ having a left adjoint, $B$ is given by $F(M)$, with left $M$-action induced by the left $M$-action on $M$.
\end{proof}

\begin{dfn}
\label{dfn:flats}
Let $B$ be a non-empty left $M$-set, and let $F : \Setswith{M} \to \Set$ be the functor $A \mapsto A \otimes_M B$. By Proposition \ref{prop:connected-objects}, $F$ preserves the terminal object if and only if $B$ is indecomposable. $F$ preserves arbitrary limits if and only if it is indecomposable projective (this follows from e.g.\ \cite[Section 4]{TDMA}). More generally, we say that $B$ is:
\begin{itemize}
\item \textbf{monomorphism-flat} if $F$ preserves monomorphisms.
\item \textbf{finitely product-flat} if $F$ preserves finite products.
\item \textbf{product-flat} if $F$ preserves products.
\item \textbf{equalizer-flat} if $F$ preserves equalizers.
\item \textbf{pullback-flat} if $F$ preserves pullbacks.
\item \textbf{flat} if $F$ preserves finite limits.
\end{itemize}
\end{dfn}

The definitions for right $M$-sets are analogous.

It is \textit{very important} to note that this naming system differs from the naming conventions in semigroup theory literature, notably that of Bulman-Fleming and Laan in \cite{flatness}. Our terminology is the same when it comes to `finitely product-flat', `equalizer-flat' and `pullback-flat'. However, what we call `monomorphism-flat' is called `flat' in their paper. Our justification for this departure is that our naming system aligns more closely with that in the category theory literature. More precisely:
\begin{dfn}[{\cite[VII.6, Definition 2]{MLM}}]
\label{dfn:filtering}
Recall that a functor $F: \Ccal \to \Set$ is called \textbf{flat} (or \textit{filtering}) if:
\begin{enumerate}
\item $F(C)\neq \varnothing$ for some object $C$ of $\Ccal$;
\item for elements $a \in F(A)$ and $b \in F(B)$ there is an object $C$, morphisms $f : C \to A$ and $g : C \to B$, and an element $c \in F(C)$ such that $F(f)(c) = a$ and $F(f)(c) = b$;
\item for morphisms $f,g : B \to A$ in $\Ccal$ and $b \in F(B)$ such that $F(f)(b) = F(g)(b)$, there is a morphism $h : C \to B$ and an element $c \in F(C)$ such that $fh = gh$ and $F(h)(c) = b$. 
\end{enumerate}
\end{dfn}

We then have the following correspondence between our definition of flatness for $M$-sets and the notion of flatness for functors.

\begin{proposition} \label{prop:flat-functor-categorically-flat}
Consider a monoid $M$ and a left $M$-set $B$. The following are equivalent.
\begin{enumerate}
\item $B$ is flat, i.e.\ $- \otimes_M B : \Setswith{M} \to \Set$ preserves finite limits.
\item $B$ is flat as a functor $M \to \Set$.
\item $- \otimes_M B$ is the inverse image part of a point: a geometric morphism $\Set \to \Setswith{M}$.
\end{enumerate}
When $B$ is equipped with a right $N$-action, we can replace $\Set$ with $\Setswith{N}$ in this statement.
\end{proposition}
\begin{proof}
($1 \Leftrightarrow 2$) See e.g.\ \cite[VII, Theorem 3]{MLM}. \\
($2 \Leftrightarrow 3$) This appears as \cite[Theorem VII.7.2]{MLM}, an extension of which is also referred to as Diaconescu's theorem by Johnstone in \cite[Theorem B3.2.7]{Ele}.
\end{proof}

Flat left $M$-sets don't just give us the points of $\Setswith{M}$; by considering the value of these functors at $M$, we see that the natural transformations
\begin{equation*}
\begin{tikzcd}
-\otimes_M B \ar[r] & -\otimes_M B'
\end{tikzcd}
\end{equation*}
correspond bijectively to the morphisms of left $M$-sets $B \to B'$. Moreover, the \textbf{essential points}, for which the inverse image functor has a left adjoint, correspond by \cite[Section 4]{TDMA} to the indecomposable projective left $M$-sets, which are indexed by the idempotents of $M$. In summary:
\begin{crly}
The category of points of $\Setswith{M}$ is equivalent to the category of flat left $M$-sets and left $M$-set homomorphisms; the full subcategory of essential points is equivalent to the full subcategory on the indecomposable projective left $M$-sets.
\end{crly}

It follows from the definitions that, for example, flat left $M$-sets are pullback-flat. However, other interactions between the different notions of flatness are not so clear. We present some general facts which simplify the situation. These are (by now) well-known in category theory literature thanks to  authors such as Freyd and Scedrov in \cite{freyd-scedrov}, but were reached independently by semigroup theorists such as Bulman-Fleming in \cite{bulman-fleming}. We reproduce proofs here anyway.
\begin{proposition} \label{prop:flatness-properties}
Let $F: \Setswith{M} \to \Setswith{N}$ be a functor. Then:
\begin{enumerate}
\item Suppose that $F$ is nontrivial (i.e.\ $F(A) \neq \varnothing$ for at least one $A$).  If $F$ preserves binary products, then it also preserves the terminal object; thus in order for a left $M$-set $B$ to be finitely product-flat, it suffices that $B \neq \varnothing$ and $- \otimes_M B$ preserves binary products.
\item If $F$ preserves pullbacks, then it also preserves equalizers, so a pullback-flat object is equalizer-flat.
\item If $F$ preserves pullbacks and the terminal object, then it preserves all finite limits, so an object is flat if and only if it is indecomposable and pullback-flat.
\end{enumerate}
\end{proposition}
\begin{proof} \ 
\begin{enumerate} 
\item If $F$ preserves binary products, then in particular the natural map
\begin{gather*}
F(1 \times 1) \longrightarrow F(1)\times F(1) \\
x \mapsto (x,x)
\end{gather*}
is an isomorphism. Surjectivity shows that $F(1)$ has at most one element. If $F(1) \simeq \varnothing$ then $F$ is trivial, so $F(1) \simeq 1$ is a singleton.
\item This is a special case of \cite[Chapter 1, \S 1.439]{freyd-scedrov}. Suppose that $F$ preserves pullbacks, and consider a diagram
\begin{equation*}
\begin{tikzcd}
A \ar[r,"{f}",shift left=1] \ar[r,"{g}"',shift right=1] & B
\end{tikzcd}
\end{equation*}
Then we can rewrite the equalizer $E$ of this diagram as a pullback:
\begin{equation*}
\begin{tikzcd}
E \ar[r] \ar[d] \ar[dr, phantom, "\lrcorner", very near start] & B \ar[d,"\delta"] \\
A \ar[r,"{(f,g)}"'] & B \times B 
\end{tikzcd}
\end{equation*}
where $\delta$ is the diagonal map. Pullbacks are preserved by $F$, so 
\begin{equation*}
\begin{tikzcd}
F(E) \ar[r] \ar[d] \ar[dr, phantom, "\lrcorner", very near start] & F(B) \ar[d,"F\delta"'] \ar[r, "\sim"'] \ar[dr, phantom, "\lrcorner", very near start] & F(B) \ar[d, "\delta"'] \\
F(A) \ar[r,"{(Ff,Fg)}"] & F(B) \times_{F(1)} F(B) \ar[r, hook] & F(B) \times F(B)
\end{tikzcd}
\end{equation*}
is a composite of pullback squares; the right hand one being a pullback is a consequence of the fact that the diagonal $\delta$ of $F(B)$ in $\Setswith{N}$ factors through $F\delta$ by the universal property of $F(B) \times F(B)$. It follows that $F(E)$ can be identified with the equalizer of the diagram
\begin{equation*}
\begin{tikzcd}
F(A) \ar[r,"{F(f)}",shift left=1] \ar[r,"{F(g)}"',shift right=1] & F(B)
\end{tikzcd},
\end{equation*}
so $F$ preserves equalizers.
\item Suppose that $F$ preserves pullbacks and the terminal object. A binary product can be seen as a pullback over the terminal object, so $F$ preserves binary products (and in fact all finite products). Moreover, $F$ preserves equalizers by (2). Any functor preserving finite products and equalizers preserves all finite limits, by a similar argument to the above after expressing a given finite limit as an equalizer of morphisms between finite products. 
\end{enumerate}
\end{proof}

In the category of sets, since arbitrary coproducts are pseudo-filtered in the sense of \cite{bjerrum}, they commute with connected limits, such as pullbacks and equalizers. The same holds in the category of $M$-sets for $M$ a monoid, since colimits and limits are computed on underlying sets. So we get the following:
\begin{corollary}[{cf.~\cite[III.3.9]{MAC}}]
\label{crly:eqpb-flat}
Let $B = \bigsqcup_{i \in I} B_i$ be a right $M$-set, with each $B_i$ indecomposable. Then:
\begin{enumerate}
\item $B$ is equalizer-flat if and only if $B_i$ is equalizer-flat for all $i \in I$;
\item $B$ is pullback-flat if and only if $B_i$ is pullback-flat for all $i \in I$.
\end{enumerate}
Since each of the $B_i$ is indecomposable, it follows that $B$ is pullback-flat if and only if $B_i$ is flat for all $i \in I$, so that pullback-flat $M$-sets are precisely the $M$-sets that can be written as disjoint unions of flat $M$-sets.
\end{corollary}

\section{Properties of the Global Sections Morphism}
\label{sec:1down}

\subsection{Groups and Atomicity}

The first property we study is expressed in terms of the logical structure of the toposes involved.
\begin{dfn}[{Cf. \cite[C3.5]{Ele}}]
A geometric morphism is \textbf{atomic} if its inverse image functor is logical. For a general geometric morphism this implies local connectedness. We say a Grothendieck topos is \textbf{atomic} if its global sections geometric morphism to $\Set$ is.
\end{dfn}

For Grothendieck toposes with enough points, atomicity coincides with a property of the internal logic of the topos.
\begin{dfn}
A topos $\Ecal$ is \textbf{Boolean} if its subobject classifier is an internal Boolean algebra, or equivalently if every subobject of an object of $\Ecal$ has a complement.
\end{dfn}

\begin{xmpl}
For a sober topological space $X$, $\Sh(X)$ has enough points, so we have by \cite[Lemma C3.5.3]{Ele} that $\Sh(X)$ is atomic if and only if it is Boolean, if and only if $X$ is a discrete topological space.
\end{xmpl}

As already observed in \cite[Corollary 7.3]{TDMA}, we have the following:
\begin{thm}[Conditions for $\Setswith{M}$ to be Boolean]
\label{thm:atomic}
Let $M$ be a monoid. The following are equivalent.
\begin{enumerate}
\item $\Setswith{M}$ is atomic.
\item $\Delta$ preserves the subobject classifier
\item $M$ is a group.
\item $\Setswith{M}$ is Boolean.
\end{enumerate}
\end{thm}
\begin{proof}
($1 \Leftrightarrow 2$) By definition, since we have already seen that $\Delta$ preserves exponentials.

($2 \Rightarrow 3$) The subobject classifier in $\Set$ is the two element set, often denoted $\Omega = \{\top,\bot\}$ so that the canonical subobject is the inclusion of the singleton $\{\top\}$. But then $\Omega \cong 1\sqcup1$, and so the subobject classifier is preserved by $\Delta$ if and only if $\Omega_{\Setswith{M}}$ has an underlying set with two elements, which forces every principal ideal in $M$ to contain the identity element (since the ideal must be all of $M$), and hence every element of $M$ has a right inverse, making $M$ a group.

($3 \Rightarrow 4$) If $M$ is a group, every subobject of an $M$-set is a union of orbits; the remaining orbits form the complementary subobject.

($4 \Leftrightarrow 1$) Since $\Setswith{M}$ has enough points (having a canonical surjective point), by \cite[Corollary C3.5.2]{Ele} being Boolean is equivalent to being atomic over $\Set$.
\end{proof}

It should be noted that many of the remaining properties explored in this paper either trivially hold for groups or are incompatible with the property of being a group (for example, a group with any kind of absorbing element is automatically trivial).

\subsection{Right-factorable Finite Generation and Strong Compactness}

A Grothendieck topos is called \textbf{compact} if its geometric morphism to $\Set$ is \textbf{proper}, which is to say that it preserves filtered colimits of subterminal objects. Since any hyperconnected geometric morphism is proper (cf. \cite[C3.2.13]{Ele}), $\Setswith{M}$ is always compact. However, the stronger notion of being \textbf{strongly compact} is not always satisfied.

\begin{dfn}
A Grothendieck topos $\Ecal$ is called \textbf{strongly compact} if its global sections functor $\Gamma = \Hom_{\Ecal}(1,-)$ preserves filtered colimits. More generally, a geometric morphism $f : \Fcal \to \Ecal$ is called \textbf{tidy} if $f_*$ preserves filtered colimits. Thus a Grothendieck topos is strongly compact if and only if the global sections geometric morphism is tidy.
\end{dfn}

\begin{example}
For a topological space $X$, $\Sh(X)$ is a compact topos if and only if $X$ is a compact space in the usual sense. In \cite[Example C3.4.1(a)]{Ele}, Johnstone gives an example of a compact space $X$ such that $\Sh(X)$ is not strongly compact: $X$ consists of two copies of the unit interval $[0,1]$ with the respective copies of the open interval $(0,1)$ identified, as sketched below.
\begin{center}
\begin{tikzpicture}
\coordinate (A) at (0,0);
\coordinate (B) at (2,0);
\draw [-,thick] (A) -- (B);
\node[draw,circle,inner sep=1pt,fill] at (-.1,.1) {};
\node[draw,circle,inner sep=1pt,fill] at (-.1,-.1) {};
\node[draw,circle,inner sep=1pt,fill] at (2.1,.1) {};
\node[draw,circle,inner sep=1pt,fill] at (2.1,-.1) {};
\end{tikzpicture}
\end{center}
On the other hand, if $X$ is a spectral space or coherent space\footnote{Not to be confused with the coherence spaces appearing in linear logic.} (in the terminology of Hochster \cite{hochster} or Johnstone \cite{Ele} respectively), having a base of compact open sets stable under finite intersections, then $\Sh(X)$ is strongly compact by \cite[Proposition C4.1.13 and Corollary C4.1.14]{Ele}. In particular, $\Sh(X)$ is strongly compact if $X$ is a zero-dimensional compact Hausdorff space or if $X$ is the Zariski spectrum of a commutative ring. More generally, Johnstone shows in \cite[Corollary C4.1.14]{Ele} that any compact Hausdorff space is strongly compact.
\end{example}

In the case where $M$ is a group, we have the following result, which appears in \cite[Example C3.4.1(b)]{Ele}:
\begin{prop} \label{prop:strongly-compact-groups}
For $G$ a group, the topos $\Setswith{G}$ is strongly compact if and only if $G$ is finitely generated.
\end{prop}

We will need the following definitions in our characterisation of the monoids $M$ such that $\Setswith{M}$ is strongly compact.

\begin{definition}
Let $M$ be a monoid, and let $S \subseteq M$ be a subset. We say that $S$ is \textbf{right-factorable} if whenever $x \in S$ and $y \in M$ with $xy \in S$, we have $y \in S$. Dually, we may call a subset $S \subseteq M$ left-factorable if whenever $x \in M$ and $y \in S$ with $xy \in S$, we have $x \in S$.
\end{definition}

The above definitions are related to the two-out-of-three property in category theory, in the sense that if we view $M$ as a category with one object, then a submonoid of $M$ has the two-out-of-three property precisely if it is both right-factorable and left-factorable. If $M$ is a commutative, cancellative monoid, then a submonoid $S \subseteq M$ is right-factorable (or equivalently, left-factorable) if and only if $S$ is a \emph{saturated monoid} in the sense of Geroldinger and Halter-Koch \cite{geroldinger-survey}, i.e.\ $S = \mathbf{q}(S) \cap M$, where $\mathbf{q}(S)$ denotes the groupification of $S$.

If $(M_i)_i$ is a family of right-factorable subsets of $M$, then the intersection $\bigcap_{i} M_i$ is also right-factorable (and similarly for families of left-factorable subsets). So we can define the following:

\begin{definition}[{See also \cite[Section 2]{kobayashi}}]
Let $M$ be a monoid, and $S \subseteq M$ a nonempty subset. As is standard, we write $\langle S \rangle_M$ for the submonoid of $M$ generated by $S$. We define $\langle{S}\rangle\rangle_M$ to be the smallest right-factorable subset of $M$ that contains $S$, and call this the \textbf{right-factorable submonoid generated by $S$}; the extra bracket on the right is intended to evoke the asymmetric extra property this submonoid satisfies compared with $\langle S \rangle_M$. We say that $S$ \textbf{right-factorably generates} $M$ if $\langle{S}\rangle\rangle_M = M$. We call $M$ \textbf{right-factorably finitely generated} if there is a finite subset $S \subseteq M$ such that $S$ right-factorably generates $M$. Dually, we can define $\langle\langle{S}\rangle_M$ to be the left-factorable submonoid generated by $S$, i.e.\ the smallest left-factorable submonoid containing $S$.
\end{definition}

The properties of being right-factorable and right-factorably finitely generated appear in the semigroup literature under the name \textit{left unitary} and \textit{left unitarily finitely generated} see e.g.\ \cite[p.63]{howie-95}, \cite[Definition 4.38]{MAC}. We prefer to employ terms which convey the elementary notion in terms of elements of the monoid, viewed as morphisms.

The submonoid right-factorably generated by any subset can be computed inductively; a functionally identical construction is given by Kobayashi in \cite[before Proposition 2.4]{kobayashi}.
\begin{lemma}
\label{lem:construct}
Given an non-empty subset $S = S_0 \subseteq M$, inductively define for $i \in \Nbb$ the subset $S_{i+1}\subseteq M$ to be $\{m \in M \mid \exists t \in \langle S_i \rangle_M, \, tm \in \langle S_i \rangle_M\}$. Then $\langle S \rangle\rangle_M$ is precisely the subset $\bigcup_{i \in \Nbb} S_i$.
\end{lemma}
\begin{proof}
We show first inductively that $S_i \subseteq \langle S \rangle\rangle_M$ for each $i$. By definition, $S_0 \subseteq \langle S \rangle\rangle_M$. Given that $S_{i-1} \subseteq \langle S \rangle\rangle_M$, to be closed under composition, $\langle S \rangle\rangle_M$ must certainly contain $\langle S_{i-1} \rangle_M$. It follows from the definition of $S_i$ that any right-factorable subset containing $S_{i-1}$ must contain $S_i$, as claimed.

Conversely, $\bigcup_{i \in \Nbb} S_i$ is right-factorable: given $x,y$ in the union, they are both contained in some $S_i$ for some sufficiently large index $i$, and hence $xy \in S_{i+1} \subseteq \bigcup_{i \in \Nbb} S_i$. Similarly, given $t \in S_i$, $m \in M$, with $tm \in \bigcup_{i \in \Nbb} S_i$, we have $tm \in S_j$ for some $j \geq i$, and hence $m \in S_{j+1}$. Thus we are done.
\end{proof}

\begin{rmk}
\label{rmk:leftOre}
If $S_0$ is a subset, or more particularly a submonoid, satisfying the left Ore condition (the dual of Definition \ref{dfn:rOre} below), we find that $S_2 = S_1$: for $m \in S_2$ we have $t = t_1 \cdots t_k$ and $s = s_1 \cdots s_l$ in $\langle S_1 \rangle_M$ such that $tm = s$ and moreover $x_1, \dotsc, x_k$, $y_1, \dotsc, y_l$ in $S_0$ with $x_it_i$ and $y_js_j$ in $\langle S_0 \rangle_M$ for each $i$ and $j$. The left Ore condition allows us to inductively construct from these an element $z \in S_0$ such that $zs$ and $zt$ are both members of $\langle S_0 \rangle_M$, whence $m \in S_1$. Thus the construction described in Lemma \ref{lem:construct} terminates after a single step.
\end{rmk}

\begin{lemma} \label{lem:quotient-rffg}
Let $S \subseteq M$ be a non-empty subset and let $\sim_S$ be the right congruence generated by the relations $x \sim 1$ for $x \in S$. Then:
\begin{equation*}
\langle{S}\rangle\rangle_M = \{ m \in M : m \sim_S 1 \}.
\end{equation*}
\end{lemma}
\begin{proof}
$\subseteq$. It is enough to show that $\{ m \in M : m \sim_S 1 \}$ is a right-factorable submonoid. If $m \sim_S 1$ and $m' \sim_S 1$ then $mm' \sim_S m'$, and by transitivity $mm' \sim_S 1$. Further, if $m \sim_S 1$ and $mm' \sim_S 1$ then $mm' \sim_S m'$ so it follows that $m' \sim_S 1$.

$\supseteq$. Take $m \in M$. If $m \sim_S 1$ then there must be some one-step relation of the form $m' \sim xm'$ with $xm' \in \langle{S}\rangle\rangle_M$ and $m' \notin \langle{S}\rangle\rangle_M$ or $xm' \notin \langle{S}\rangle\rangle_M$ and $m' \in \langle{S}\rangle\rangle_M$. Both cases lead to a contradiction.
\end{proof}

With the above definitions, it is possible to generalise Proposition \ref{prop:strongly-compact-groups} to arbitrary monoids $M$, by adapting Johnstone's proof \cite[C3.4.1(b)]{Ele}.
\begin{theorem}[Conditions for $\Setswith{M}$ to be strongly compact] \label{thm:strongly-compact}
For a monoid $M$, the following are equivalent:
\begin{enumerate}
\item $\Gamma$ preserves filtered colimits, i.e.\ $\Setswith{M}$ is strongly compact.
\item The terminal right $M$-set $1$ is finitely presentable.
\item $M$ is right-factorably finitely generated.
\end{enumerate}
\end{theorem}
\begin{proof}
($1 \Leftrightarrow 2$) Note that $\Gamma = \HOM_M(1,-)$. The equivalence is then by definition, see Definition \ref{dfn:projectives}. See also Proposition \ref{prop:finitely-presentable-topos-vs-algebra} which states that the topos-theoretic and algebraic definitions of ``finitely presentable'' agree.

($1 \Rightarrow 3$) For $S \subseteq M$ a finite subset, let $\sim_S$ be the right congruence as defined in Lemma \ref{lem:quotient-rffg}. Clearly $1 = \varinjlim_S M/\!\!\sim_S\,$, where we take the filtered colimit over all finite subsets of $M$. Now consider the comparison map
\begin{equation*}
\beta : \varinjlim_S \,\Gamma(M/\!\!\sim_S) \longrightarrow 1
\end{equation*}
If $\Gamma$ preserves filtered colimits, then in particular $\beta$ is an isomorphism. By surjectivity of $\beta$, $M/\sim_S$ has a fixed point for some finite subset $S$ of $M$, we will denote it by $[a]$ for some representative $a \in M$. Its projection in $M/\!\!\sim_{S'}$ is again a fixpoint for $S' = S\cup \{a\}$. Now $[1]$ is a fixpoint in $M/\!\!\sim_{S'}$. But this means that $m \sim_{S'} 1$ for all $m \in M$. By Lemma \ref{lem:quotient-rffg}, $S'$ right-factorably generates $M$.

($3 \Rightarrow 1$) Take a finite subset $S \subseteq M$ such that $\langle{S}\rangle\rangle_M = M$. Consider a filtered colimit $\varinjlim_i S_i$ of right $M$-sets $S_i$. We have to show that the natural map 
\begin{equation*}
\varinjlim_i \Gamma(S_i) \longrightarrow \Gamma(\varinjlim_i S_i)
\end{equation*}
is a bijection. Injectivity is immediate from filteredness, so we prove surjectivity. Take an element $x$ in $\Gamma(\varinjlim_i S_i)$. Let $x$ be represented by an element $x_i \in S_i$. For each $s \in S$, we can find an index $j$ and a structure morphism $\phi_{ji} : S_i \to S_j$ such that $\phi_{ji}(xs) = \phi_{ji}(x)$. Since $S$ is finite we can find a common index $k$ such that $\phi_{ki}(xs) = \phi_{ki}(x)$ for all $s \in S$. Since $\phi_{ki}(x)$ is fixed by all $s \in S$, it is also fixed by all elements of $\langle{S}\rangle\rangle_M = M$: indeed, employing the inductive construction of Lemma \ref{lem:construct}, if $x$ is fixed by $S_i$, then it is fixed by $\langle S_i \rangle_M$, and given $m \in M$, $t \in \langle S_i \rangle_M$ with $tm \in \langle S_i \rangle_M$, we have $xm = xtm = x$, so $x$ is fixed by $S_{i+1}$, and hence inductively by $\bigcup_{i \in \Nbb} S_i = M$. So $x$ is represented by an element of $\Gamma(S_k)$.
\end{proof}
The (dual of the) equivalence ($2 \Leftrightarrow 3$) appears in the semigroup theory work of Dandan, Gould, Quinn--Gregson and Zenab \cite[Theorem 3.10]{dandan-gould-quinn-gregson-zenab}, amongst some other equivalent conditions.

\begin{example} \label{xmpl:rwg} \ 
If $M$ has a right-absorbing element $r$ in the sense of Definition \ref{dfn:absorb} then $M$ is right-factorably generated by $\{r\}$ (since $r\cdot m = r$ for all $m \in M$). In particular $M$ is right-factorably finitely generated, so $\Setswith{M}$ is strongly compact. Given a non-empty set $S$, the monoid $M=\End(S)$ of (total) functions $S \to S$ has a right-absorbing element for each element $s \in S$ given by the constant function at $s$, so $M$ is right-factorably finitely generated by the above.
\end{example}

\begin{example}
We already saw that for a group $G$, the topos $\Setswith{G}$ is strongly compact if and only if $G$ is finitely generated. We can use this to produce some more examples of monoids $M$ such that $\Setswith{M}$ is \textit{not} strongly compact. Let $G$ be a group and let $M \subseteq G$ be a submonoid such that
\begin{equation*}
G = \{ a^{-1}b : a, b \in M \}.
\end{equation*}
Note that $\langle{M}\rangle\rangle_G = G$. Now if $S \subseteq M$ is a finite set right-factorably generating $M$, then $\langle{S}\rangle\rangle_G$ contains $M$ and is closed under right factors in $G$, so $\langle{S}\rangle\rangle_G = G$. It follows that $G$ is finitely generated. Therefore, for the following monoids $M$ the topos $\Setswith{M}$ is not strongly compact:  
\begin{itemize}
\item The monoid $\mathbb{N}^{\times}$ of nonzero natural numbers under multiplication;
\item The monoid $R-\{0\}$ for $R$ an infinite commutative ring without zero-divisors (noting that finitely generated fields are finite);
\item The monoid of non-singular $n\times n$ matrices over $R$, for $R$ an infinite commutative ring without zero divisors.
\end{itemize}
\end{example}

The topos $\Setswith{\mathbb{N}^{\times}}$ is (the underlying topos of) the Arithmetic Site by Connes and Consani \cite{connes-consani}. It follows that this topos is not strongly compact. In \cite{connes-consani-geometry-as}, the monoid $\mathbb{N}^{\times}_0 = \mathbb{N}^{\times} \cup \{0\}$ is considered as well, and in this case the topos $\PSh(\mathbb{N}^{\times}_0)$ is strongly compact, by Example \ref{xmpl:rwg}.

Note that the equivalence ($1 \Leftrightarrow 2$) in Theorem \ref{thm:strongly-compact} directly generalises to geometric morphisms $f: \Setswith{M} \to \Setswith{N}$ induced by a tensor-hom adjunction as in Proposition \ref{prop:flat-functor-categorically-flat}. So we have the following equivalence:

\begin{crly}
\label{crly:tidiness}
Let $M$ and $N$ be monoids, and let $B$ be a left-$M$-right-$N$-set. Suppose that the action of $M$ is flat, so that the adjunction induced by $B$ is a geometric morphism $f : \Setswith{N} \to \Setswith{M}$. Then $f$ is tidy if and only if $B$ is finitely presented as right $N$-set.
\end{crly}

\subsection{Right absorbing elements and Localness}

The properties that we explore in this subsection are derived from the concept of localness.
\begin{dfn}[see {\cite[C3.6]{Ele}}]
\label{dfn:local}
A Grothendieck topos is called \textbf{local} if its global sections functor has a right adjoint. More generally, a geometric morphism $f: \Fcal \to \Ecal$ is \textbf{local} if its direct image functor $f_*$ has an $\Ecal$-indexed right adjoint.
\end{dfn}

To get some geometrical intuition for this concept, we mention the following criterion for localness for the topos of sheaves on a topological space.

\begin{proposition}[{Cf. \cite[C1.5, pp.~523]{Ele}}]
\label{prop:localspace}
Let $X$ be a sober topological space. Then the following are equivalent:
\begin{enumerate}
\item $X$ has a \textbf{focal point}, i.e.\ a point such that the only open set containing it is $X$ itself; being sober makes such a point unique if it exists.
\item $\Sh(X)$ is a local topos.
\end{enumerate}
\end{proposition}

In particular, if $R$ is a commutative ring then the topos of sheaves on $\mathrm{Spec}(R)$ (with the Zariski topology) is a local topos if and only if $R$ has a unique maximal ideal, or in other words if and only if $R$ is a local ring.

\begin{dfn}
\label{dfn:absorb}
An element $m$ of a monoid $M$ is called \textbf{right absorbing} if it absorbs anything on its right, so $mn=m$ for all elements $n \in M$; \textbf{left absorbing} elements are defined dually. An element which is both left and right absorbing is called a \textbf{zero element}, because the $0$ of a commutative ring has this property in the ring's multiplicative monoid. This final case is of the broadest interest, but since left and right absorbing elements manifest themselves very differently in $\Setswith{M}$, we study them independently. Note that if a monoid has both left absorbing element $l$ and a right absorbing element $r$, then $l=r$ is a zero element (which is automatically unique).
\end{dfn}

This convention of handedness of absorbing elements is somewhat arbitrary, since one could alternatively take `right absorbing' to mean `absorbs all elements when multiplied on the right'. Both conventions appear in semigroup and semiring literature; we follow \cite{golan}.

A result allowing us to compare $\Gamma$ and $C$ which we have not yet had cause to introduce is the following, appearing in a more general form in Johnstone's work, \cite[Corollary 2.2(a)]{PLC}:
\begin{lemma}
\label{lem:alpha}
Let $f:\Fcal \to \Ecal$ be a connected essential geometric morphism, so the unit $\eta$ of $(f^* \dashv f_*)$ and the counit $\delta$ of $(f_! \dashv f^*)$ are isomorphisms. Write $\epsilon$ for the counit of the former and $\nu$ for the unit of the latter adjunction. Then there is a canonical comparison natural transformation $\alpha:f_* \to f_!$ which can be expressed as either $f_!\epsilon \circ (\delta_{f_*})^{-1}: f_* \to f_!f^*f_* \to f_!$ or as $(\eta_{f_!})^{-1} \circ f_*\nu: f_* \to f_*f^*f_! \to f_!$. 
\end{lemma}
A more concrete description of this transformation $\alpha:\Gamma \to C$ for $\Setswith{M}$ is that it sends a fixed point of a right $M$-set $X$ to the connected component of $X$ containing it.

From the same paper, we obtain the following term.
\begin{dfn}
\label{dfn:PLC}
A connected, locally connected geometric morphism is said to be \textbf{punctually locally connected} if the natural transformation of Lemma \ref{lem:alpha} is epic. A connected, locally connected topos is called \textbf{punctually locally connected} if the global sections geometric morphism is punctually locally connected, which can be interpreted in the case of $\Setswith{M}$ as the statement `every component has at least one fixed point'. 
\end{dfn}

In \cite{lawvere-menni}, Lawvere and Menni call a geometric morphism \textbf{pre-cohesive} if it is local, hyperconnected and essential, such that $f_!$ preserves finite products; it was shown in \cite{PLC} that the global sections geometric morphism of a Grothendieck topos satisfies these properties if and only if it is punctually locally connected.

\begin{rmk}
Unlike many of the other properties in this paper, punctual local connectedness is not a geometric property: as shown in \cite[Proposition 1.6]{PLC}, it forces a connected, locally connected geometric morphism to be hyperconnected, which means that the only sober space $X$ with $\Sh(X)$ punctually locally connected is the one-point space. 
\end{rmk}

\begin{lemma}
\label{lem:rfixed}
Suppose $M$ is a monoid with a right absorbing element $r$. Then for any right $M$-set $A$, we have $\Gamma(A) = Ar$.
\end{lemma}
\begin{proof}
Clearly any element of $A$ of the form $ar$ with $r \in R$ is fixed by the action of $M$. Conversely, if $a$ is fixed by the action of $M$ then $ar = a$ for any $r \in R$. Thus every fixed point is in $Ar$.
\end{proof}

\begin{thm}[Conditions for $\Setswith{M}$ to be local]
\label{thm:rabsorb}
Let $M$ be a monoid, $\Gamma:\Setswith{M} \to \Set$ the global sections functor, $C:\Setswith{M} \to \Set$ the connected components functor and $\alpha:\Gamma \to C$ the natural transformation of Lemma \ref{lem:alpha}. The following are equivalent:
\begin{enumerate}
	\item $M$ has a right absorbing element.
	\item $\Gamma$ preserves epimorphisms, i.e.\ the right $M$-set $1$ is projective.
	\item $\Gamma$ has a right adjoint, $\gamma$. That is, $\Setswith{M}$ is local over $\Set$.
	\item $\alpha:\Gamma \to C$ is epic. That is, $\Setswith{M}$ is punctually locally connected (equivalently, pre-cohesive). \label{item:rab6}
	\item $C$ is full. \label{item:rab7}
	\item $\Gamma$ reflects the initial object, i.e.\ every non-empty right $M$-set has a fixed point.
	\item The category of points of $\Setswith{M}$ has an initial object.
\end{enumerate}
\end{thm}
\begin{proof}
($1 \Leftrightarrow 2$) Since $1$ is indecomposable, it is projective if and only if there is an idempotent $e \in M$ such that $1 = eM$, but the latter equality holds if and only if $e$ is a right absorbing element.

($2 \Leftrightarrow 3$) $\Gamma = \HOM_M(1,-)$ has a right adjoint if and only if it preserves colimits, by the Special Adjoint Functor Theorem. This is the case if and only if $1$ is (indecomposable) projective.

($2 \Rightarrow 4$) Since this geometric morphism is hyperconnected and locally connected, it is punctually locally connected if (and only if) $\Gamma$ preserves epimorphisms, by \cite[Lemma 3.1(ii)]{PLC}. Alternatively, for $X$ a right $M$-set, observe that the unit $X \to \Delta C(X)$ is epic. If $\Gamma$ preserves epimorphisms, then in particular $\Gamma(X) \to \Gamma \Delta C(X) \cong C(X)$ is an epimorphism.

($4 \Leftrightarrow 5$)${}^*$ By the axiom of choice, any epimorphism in $\Set$ (in particular any component of $\alpha$) splits. Given a function $g:C(X) \to C(Y)$, we therefore obtain a morphism $X \to Y$ lifting $g$ by sending every $x \in X$ to the fixed point $\alpha_Y^{-1}\circ g([x])$, where $\alpha_Y^{-1}:C(Y) \to \Gamma(Y)$ is a splitting for $\alpha_Y$ and $[x]$ is the connected component of $X$ containing $x$, so $C$ is full. Conversely, by standard adjunction arguments, $C$ is full if and only if the unit $\nu$ of the adjunction $(C \dashv \Delta)$ is componentwise a split epimorphism. $C$ being full therefore makes each component of $\nu$, and hence of $\Gamma \nu$ and $\alpha = (\eta_{C})^{-1} \circ \Gamma\nu$ a split epimorphism.

($4 \Rightarrow 6$) Since $C$ reflects the initial object by inspection, $\alpha:\Gamma \to C$ being an epimorphism ensures that $\Gamma(X)$ is non-empty whenever $C(X)$ is.

($6 \Rightarrow 1$) Consider $M$ as a right $M$-set under multiplication. 

($3 \Rightarrow 7$) This holds for any topos by \cite[C3.6, Theorem 3.6.1]{Ele}.

($7 \Rightarrow 1$) Let $A$ be an initial object in the category of points of $\Setswith{M}$, i.e.\ in the category of flat left $M$-sets. Let $f : A \to M$ be the unique morphism to $M$. Note that flatness implies indecomposability, so in particular $A$ is non-empty. This means we can take a morphism $g: M \to A$. Because $A$ is initial, $gf$ is the identity, so $A$ is a retract of $M$, i.e.\ $A = Me$ for some idempotent $e \in M$. The morphisms of left $M$-sets $Me \to M$ correspond to the elements of $eM$, so if $A$ is initial, then $e$ is right absorbing.
\end{proof}

The equivalence between conditions $(1)$ and $(8)$ of Theorem \ref{thm:rabsorb} is shown by Bulman-Fleming and Laan in \cite[Corollary 3.6]{flatness}, and by Kilp \textit{et al.} in \cite[Proposition III.17.2(4)]{MAC}. As the proof suggests, several of the conditions are shown to be general consequences of one another in \cite{PLC}.

\begin{remark}
In the above theorem, we could replace $(3)$ with the statement $\Gamma$ preserves pushouts, or coequalizers, or reflexive coequalizers. Indeed, because $1$ is indecomposable, $\Gamma = \HOM_M(1,-)$ preserves coproducts, so if it preserves pushouts then it also preserves coequalizers (in particular reflexive coequalizers). Conversely, any epimorphism can be written as a reflexive coequalizer, so if $\Gamma$ preserves reflexive coequalizers then it preserves epimorphisms as well.
\end{remark}

\subsection{The Right Ore Condition, Preservation of Monomorphisms and de Morgan Toposes}

We can view the preceding sections as investigations of `projectivity' properties of the terminal right $M$-set, which generally corresponded to properties of $\Gamma$. We move on in this section to the examination of the `flatness' properties of terminal right $M$-set, corresponding to properties of $C$. The first and weakest of these is monomorphism-flatness of the terminal left $M$-set; one equivalent property to this is the dual of one from the last section:
\begin{dfn}
\label{dfn:coPLC}
Dualising Definition \ref{dfn:PLC}, we say a connected, locally connected geometric morphism $f$ is \textbf{co\-punctually locally connected} if the natural transformation $\alpha$ of Lemma \ref{lem:alpha} is a monomorphism. A Grothendieck topos is called \textbf{co\-punctually locally connected} if the global sections geometric morphism is copunctually locally connected, which can be interpreted in the case of $\Setswith{M}$ as `every component has at most one fixed point'.
\end{dfn}

\begin{example}
Like punctual local connectedness, copunctual local connectedness is not a geometric property. Suppose $X$ is a connected, locally connected sober space. Consider the global sections geometric morphism $f : \Sh(X) \to \Set$. Viewing the objects of $\Sh(X)$ as local homeomorphisms, the map $\alpha_E : f_*(E) \to f_!(E)$ sends each global section $s$ to the unique connected component of $E$ that contains the image $s(X)$. If $X$ has an open subset $U$ not equal to the empty set or all of $X$, we can construct a local homeomorphism $\pi:E \to X$ by taking $E$ to be the quotient of the disjoint union of two copies of $X$ which identifies the two copies of $U$, and take $\pi$ to be the natural projection map. This $E$ is connected since $X$ is, but has exactly $2$ global sections, so $\alpha_E$ fails to be monic. Thus $\Sh(X)$ is colocally punctually connected if and only if $X$ is the one point space.
\end{example}

Another equivalent property relates to the internal logic of $\Setswith{M}$.
\begin{dfn}
A topos $\Ecal$ is said to be \textbf{de Morgan} if its subobject classifier is de Morgan as an internal Heyting algebra. Equivalently, this says that for any subobject $A$ of an object $B$ of $\Ecal$, we have $(A \rightarrow 0) \cup ((A \rightarrow 0) \rightarrow 0) = B$ in the Heyting algebra of subobjects of $B$.
\end{dfn}

\begin{example}
Let $X$ be a topological space. Then from \cite[Example D4.6.3(b)]{Ele} we know that $\Sh(X)$ is de Morgan if and only if $X$ is \textbf{extremally disconnected}, i.e.\ the closure of every open subset of $X$ is again open.

The name `extremally disconnected' is a bit misleading, since the existence of a dense point in $X$ (see Proposition \ref{prop:veryconnected} below) makes $X$ both connected and extremally disconnected. An explanation for this confusion is that extremal disconnectedness was originally only defined by Gleason for Hausdorff spaces in \cite{gleason}. In that situation, extremal disconnectedness is strictly stronger than total disconnectedness.

More generally, if $X = \bigsqcup_{i \in I} X_i$ is a disjoint union of a family $\{X_i\}_{i \in I}$ of irreducible topological spaces in the sense to be defined in Proposition \ref{prop:veryconnected}, then it follows that $X$ is extremally disconnected. If $X$ is a variety (with the Zariski topology), then $X$ is extremally disconnected if and only if each of its connected components is irreducible.
\end{example}

\begin{dfn}
\label{dfn:rOre}
A monoid $M$ is said to satisfy the \textbf{right Ore condition} or is \textit{left reversible} if for every $m_1,m_2 \in M$ there exists $m_1',m_2' \in M$ with $m_1m_1'=m_2m_2'$, or equivalently $m_1M \cap m_2M \neq \emptyset$. We employ the former terminology, as used by Johnstone in \cite[Example A2.1.11]{Ele}; the latter is employed by Sedaghatjoo and Khaksari in \cite{sedaghatjoo-khaksari} and by Kilp \textit{et al.} in \cite{MAC}.
\end{dfn}

\begin{thm}[Conditions for $\Setswith{M}$ to be de Morgan]
\label{thm:deMorgan}
Let $M$ be a monoid, $\Omega$ the subobject classifier of $\Setswith{M}$, $\Gamma, C: \Setswith{M} \to \Set$ the usual functors and $\alpha:\Gamma \to C$ the natural transformation from Lemma \ref{lem:alpha}. The following are equivalent:
\begin{enumerate}
	\item $C$ preserves monomorphisms, so the terminal left $M$-set is monomorphism-flat.
	\item Any two non-empty sub-$M$-sets of an indecomposable right $M$-set have non-empty intersection. \label{item:deMorgan-7}
	\item $M$ satisfies the right Ore condition.
	\item Any sub-$M$-set of an indecomposable right $M$-set is either empty or indecomposable. 
	\item $C(\Omega)$ is a two element set (combined with condition $1$, this means that $C$ preserves the subobject classifier). \label{item:deMorgan-5}
	\item $\Setswith{M}$ is de Morgan.
	\item $\alpha:\Gamma \to C$ is a monomorphism; that is, $\Setswith{M}$ is copunctually locally connected.
\end{enumerate}
\end{thm}
\begin{proof}
($1 \Rightarrow 2$) Let $A$ be indecomposable and let $A_1,A_2 \subseteq A$ be two non-empty sub-$M$-sets of $A$. Since $C(A_1\cup A_2) \hookrightarrow C(A)=1$ (and the first expression has at least one element), the two subsets must intersect.

($2 \Rightarrow 3$) Applying ($2$) to the principal $M$-sets generated by elements $m_1,m_2 \in M$ gives the right Ore condition.

($3 \Leftrightarrow 4$) Informally, given a finite zigzag connecting elements $a_1,a_2$ in an indecomposable right $M$-set $A$, the right Ore condition allows us to inductively `push out' the spans of this zigzag to obtain elements $m_1,m_2$ with $a_1\cdot m_1 = a_2 \cdot m_2$, so any sub-$M$-set containing $a_1$ intersects any containing $a_2$; it follows that any non-empty sub-$M$-set is indecomposable. Conversely, the union of a pair of principal ideals in $M$ being indecomposable means they must intersect, else we could not construct a connecting zigzag.

($4 \Rightarrow 1$) A subobject of an $M$-set $A$ is a coproduct of subobjects of the indecomposable components of $A$. In particular, ($4$) ensures that the image under $C$ of an inclusion of sub-$M$-sets is monic.

($3 \Leftrightarrow 5$) The subobject classifier $\Omega$ of $\Setswith{M}$ is the collection of right ideals of $M$. Given a non-empty ideal $I$ of $M$ and $m \in I$, we have $I \cdot m = M$, so the connected component of $\Omega$ containing $M$ also contains all of the non-empty ideals. Since $\emptyset$ is a fixed point of $\Omega$, $C(\Omega)$ has two elements if and only if $I \cdot m$ is non-empty whenever $I$ is for every $m \in M$ (otherwise $C(\Omega)$ is a singleton). Since the action of $M$ on $\Omega$ is order-preserving, it is necessary and sufficient that this is true for the principal ideals, and $m_1M \cdot m \neq \emptyset$ if and only if $m_1M \cap mM \neq \emptyset$.

($3 \Leftrightarrow 6$) By \cite[Example D4.6.3(a)]{Ele}, the topos of presheaves on a category $\Ccal$ is de Morgan if and only if $\Ccal$ satisfies the right Ore condition, which reduces to Definition \ref{dfn:rOre} when $\Ccal$ has a single object. 

($1 \Rightarrow 7$) Since the counit of a hyperconnected geometric morphism is monic \cite[Proposition A4.6.6]{Ele}, we have $\alpha_X = C\epsilon_X \circ \left(\delta_{\Gamma(X)}\right)^{-1}$ is monic when $C$ preserves monomorphisms.

($7 \Rightarrow 3$) Let $m_1,m_2 \in M$. Generate a right congruence of $M$ from the relations $m_1n \sim m_1n'$ and $m_2n \sim m_2n'$ for all $n,n' \in M$. The quotient of $M$ by this congruence is indecomposable, so has at most one fixed point since $\alpha$ is monic; in particular, the equivalence classes represented by $m_1$ and $m_2$ are fixed and so must be equal. That is, $m_1M \cap m_2M \neq \emptyset$.
\end{proof}
The equivalences ($1 \Leftrightarrow 3$) and a partial form of ($3 \Leftrightarrow 4$) appear respectively as \cite[Exercise 12.2(2) and Exercise 11.2(2)]{MAC}.

\begin{remark}
\label{rmk:sufficiently}
Following \cite{lawvere-menni}, an object $X$ of a topos $\Ecal$ is called \textbf{contractible} if $X^A$ is connected for all objects $A$ of $\Ecal$. Lawvere and Menni say a pre-cohesive topos is \textbf{sufficiently cohesive} if every object can be embedded in a contractible object. Equivalently, a pre-cohesive topos is sufficiently cohesive if and only if the subobject classifier is connected, see \cite[Proposition 4]{cohesion}. As we observed in the proof of ($3 \Leftrightarrow 5$) above, the subobject classifier of $\Setswith{M}$ is connected if and only if $M$ \textit{does not} satisfy the right Ore condition. It follows easily that $\Setswith{M}$ is sufficiently cohesive if and only if $M$ has at leas two right-absorbing elements.
\end{remark}

Since any group satisfies the right Ore condition, the equivalent properties of Theorem \ref{thm:deMorgan} are satisfied when $M$ is a group; we shall see further special cases while investigating stronger properties in subsequent sections. Since every monomorphism in a topos is regular, $C$ preserving equalizers implies $C$ preserves monomorphisms. However, preserving equalizers is a strictly stronger condition:
\begin{example}[$C$ can preserve monomorphisms but not equalizers]
Consider the commutative monoid $M = \mathbb{N}$ of natural numbers under addition. Clearly, $\mathbb{N}$ satisfies the right Ore condition, so $\Setswith{\mathbb{N}}$ is de Morgan. We prove that the functor $C$ does not preserve equalizers in this case. Consider the diagram of right $\mathbb{N}$-sets
\begin{equation} \label{eq:example-dm-not-tc}
\begin{tikzcd}
\mathbb{N} \ar[r,shift left=1,"{\id}"] \ar[r,shift right=1,"{s}"'] & \mathbb{N},
\end{tikzcd}
\end{equation}
with $s$ the successor map, i.e.\ $s(n) = n+1$. Then the equalizer of this diagram is empty. But applying the connected components functor to (\ref{eq:example-dm-not-tc}) gives
\begin{equation*}
\begin{tikzcd}
\{\ast \} \ar[r,shift left=1,"{C(\id)}"] \ar[r,shift right=1,"{C(s)}"'] & \{\ast \}
\end{tikzcd}
\end{equation*}
with $C(\id)$ and $C(s)$ both the identity map. So the equalizer of this diagram is $\{\ast\}$, which does not agree with $C(\varnothing) = \varnothing$.
\end{example}

We shall see in Theorem \ref{thm:totally} the extent to which preservation of equalizers is stronger than preservation of monomorphisms for $C$. We first examine the conditions under which $C$ preserves finite products.

\subsection{Spans and Strong Connectedness}
\label{ssec:strong-connectedness}

\begin{dfn}
A geometric morphism $f:\Fcal \to \Ecal$ is called \textbf{strongly connected} if it is locally connected and its left adjoint $f_!$ preserves finite products.
\end{dfn}

If $f$ is strongly connected, then in particular $f_!(1)=1$, so $f$ is connected. This means that $f$ is strongly connected if and only if it is connected, locally connected, such that $f_!$ preserves binary products.

One justification for this name is a geometric one, but as we shall see in Proposition \ref{prop:veryconnected}, it coincides with a stronger property, called total connectedness, for toposes of the form $\Sh(X)$. Therefore we give some justification in terms of presheaf toposes in Example \ref{xmpl:strongconnect}.

Let $\Dcal$ be a small category. Recall that colimits of shape $\Dcal$ commute with finite products in $\Set$ if and only if $\Dcal$ is a \textbf{sifted category}; we refer the reader to the survey on sifted colimits \cite{Sifted} for background on these. Concretely, siftedness may be expressed as the requirement that for each pair of objects $D_1,D_2$ in $\Dcal$, the category $\mathrm{Cospan}(D_1,D_2)$ of cospans on this pair of objects is non-empty and connected. More abstractly, siftedness may be characterised by the diagonal functor $\Dcal \to \Dcal \times \Dcal$ being a final functor.

\begin{example}
\label{xmpl:strongconnect}
The presheaf topos $\Setswith{\Dcal}$ is always locally connected. The left adjoint $f_!$ in its global sections geometric morphism sends $X:\Dcal\op \to \Set$ to its colimit. It follows by considering the terminal object of $\Setswith{\Dcal}$ that $\Setswith{\Dcal}$ is connected if and only if $\Dcal$ is connected as a category, and that $f_!$ preserves finite products if and only if $\Dcal\op$ is sifted, which is to say that the categories of spans $\mathrm{Span}(D_1,D_2)$ with $D_1,D_2 \in \Dcal$ are connected too.
\end{example}

Applying this to the case where $\Dcal$ has just one object, we arrive at the following:
\begin{proposition}[Conditions for $\Setswith{M}$ to be strongly connected]
\label{prop:strongcon}
Let $M$ be a monoid and $C:\Setswith{M} \to \Set$ the connected components functor. The following are equivalent:
\begin{enumerate}
	\item $C$ preserves finite products, i.e.\ $\Setswith{M}$ is strongly connected over $\Set$, or the terminal left $M$-set is finitely product-flat.
	\item $C$ preserves binary products.
	\item The diagonal monoid homomorphism $D: M \to M \times M$ is initial as a functor.
	\item The category $\mathrm{Span}(M)$ of spans on $M$ is connected.
	\item Any product of two indecomposable $M$-sets is indecomposable.
	\item As a right $M$-set, $M \times M$ is indecomposable. \label{item:MxMindec}
\end{enumerate}
\end{proposition}
\begin{proof}
The equivalence of $(1)$, $(2)$, $(3)$ and $(4)$ is immediate from the preceding discussion.

($2 \Rightarrow 5 \Rightarrow 6$) Given indecomposable $M$-sets $X$ and $Y$, we have $C(X \times Y) \cong C(X) \times C(Y) \cong 1$, so $X \times Y$ is indecomposable; taking $X = Y = M$ gives ($6$) as a special case of this.

($6 \Rightarrow 4$) A morphism of spans $m: (x,y) \to (xm,ym)$ is, by inspection, the action by right multiplication of $m \in M$ on the corresponding element $(x,y) \in M \times M$.
\end{proof}

\begin{example}
\label{xmpl:groupnotstrong}
Suppose $M$ is a non-trivial group. Then $C$ fails to preserve products since $M \times M$ decomposes into orbits indexed by the elements of $M$. This can also be deduced from Proposition \ref{prop:totally2} below combined with the result \cite[Scholium A2.3.9]{Ele} that a logical functor with a left adjoint which preserves finite limits must be an equivalence.
\end{example}

\begin{lemma}
\label{lem:localstrong}
Suppose $M$ is a monoid with a right absorbing element. Then $M \times M$ is indecomposable as a right $M$-set.
\end{lemma}
\begin{proof}
At the level of toposes, one can prove that any punctually locally connected geometric morphism $f$ has leftmost adjoint $f_!$ preserving products, as Johnstone does in \cite[Proposition 2.7]{PLC}. More concretely, let $r$ be a right absorbing element and $(x,y) \in M \times M$. Then $(x,y)r = (xr,yr) = (1,yr)xr$ and $(1,1)yr = (yr,yr) = (1,yr)yr$, which gives a zigzag connecting $(x,y)$ and $(1,1)$, making $M \times M$ indecomposable as required.
\end{proof}

\subsection{Preserving Exponentials}
\label{ssec:exponential}

In the context of $\Setswith{M}$ being a strongly connected topos, it makes sense to ask under which extra conditions $C$ is cartesian-closed; we explore this and some related conditions now for the sake of curiosity. First, we recall from \eqref{eq:exp1} in Section \ref{ssec:toposes} that for right $M$-sets $P$ and $Q$ we have $Q^P = \HOM_M(M \times P,Q)$ in $\Setswith{M}$, with $m \in M$ acting by left multiplication in the first argument. The comparison morphism $\theta_{P,Q}$ for $C$ sends the connected component $[g] \in C(Q^P)$ to the function sending a component $[p] \in C(P)$ to the component $[g(1,p)] \in C(Q)$, where $g:M \times P \to Q$, $p \in P$ and $g(1,p) \in Q$ are representative elements of their respective components.

For a general geometric morphism $f:\Fcal \to \Ecal$ to be locally connected, the inverse image functor not only needs a left adjoint $f_!$, but this adjoint must be \textbf{$\Ecal$-indexed}, which can be paraphrased as the condition that transposition from $\Fcal$ to $\Ecal$ across the adjunction $(f_! \dashv f^*)$ should preserve pullback squares of a suitable form. We refer the reader to \cite[Definition 1.2.1]{Cover} for a more precise statement of this, but we will only use the following fact.
\begin{fact}
Let $f:\Fcal \to \Ecal$ be a connected, locally connected geometric morphism. Then for every $Y$ in $\Fcal$, $X$ in $\Ecal$ we have:
\[f_!(Y \times f^*(X)) \cong f_!(Y) \times X \cong f_!(Y) \times f_!f^*(X),\]
naturally in $X$ and $Y$. Thus even when $f$ is not strongly connected, we can construct the canonical morphisms $\theta_{f^*(X),Y}: f_!(Y^{f^*(X)}) \to f_!(Y)^{f_!f^*(X)}$.
\end{fact}

\begin{dfn} \label{dfn:powers}
Let $f:\Fcal \to \Ecal$ be an essential geometric morphism satisfying $f_!(1) = 1$. Observe that since colimits are stable under pullback in a topos, we have $Y \times \coprod_{i=1}^n 1 = \coprod_{i=1}^n Y$ for any object $Y$, so that products of this form are preserved by $f_!$. Thus the canonical morphisms $\theta_{\left(\coprod_{i = 1}^n 1 \right),Y}: f_!(Y^{\left(\coprod_{i=1}^n 1\right)}) \to f_!(Y)^{\left(\coprod_{i=1}^n 1\right)}$ are well-defined. When they are isomorphisms for every $n$, we say $f$ is \textbf{finitely power-connected}.

Now suppose $f$ is a connected, locally connected geometric morphism. Then we say $f_!$ \textbf{preserves $\Ecal$-indexed powers}, or that $f$ is \textbf{power-connected} if $\theta_{f^*(X),Y}$ is an isomorphism for all objects $X \in \Ecal$ and $Y \in \Fcal$. Since $f^*(\coprod_{i \in I} 1) \cong \coprod_{i \in I} 1$ in $\Fcal$, this implies finite power-connectedness.

Finally, suppose $f$ is strongly connected. We say $f$ is \textbf{cartesian-closed-connected}\footnote{For want of a better name: given the negative results of this section, we have been unable to obtain enough useful intuition about this condition in order to inform the name.} if $f_!$ is cartesian-closed. By inspection, this implies being power-connected.
\end{dfn}

The first of these definitions is clearly implied by strong connectedness, since powers correspond to products in which all entries are equal. For the global sections morphisms of the toposes studied in this paper, however, it is equally strong:
\begin{crly}
The global sections morphism of $\Setswith{M}$ is finitely power-connected if and only if it is strongly connected, since $\Setswith{M}$ is strongly connected if and only if $M \times M = M^2$ is indecomposable by Proposition \ref{prop:strongcon}. See Proposition \ref{prop:veryconnected} below for this result in the case of toposes of the form $\Sh(X)$.
\end{crly}

For our investigation of power-connectedness, we begin with the case where $M$ fails to satisfy the right Ore condition. In \cite[Proposition 3.4]{sedaghatjoo-khaksari} it is noted that a further condition coinciding with the right Ore condition is the property that every indecomposable $M$-set has finite width, in the sense that there exists an upper bound on the length of the zigzag needed to connect any pair of elements.
This is formalised as follows:
\begin{dfn}[{\cite[Section 1]{sedaghatjoo-khaksari}}] 
\label{dfn:width}
Let $M$ be a monoid and $A$ a right $M$-set. We say that two elements $a,b \in A$ can be \textbf{connected by a scheme of length $n$} if we can find $s_1,\dots,s_n,t_1,\dots,t_n \in M$, $a_1,\dots,a_n \in A$ such that
\begin{equation*}
a = a_1 s_1,~ a_1t_1=a_2s_2,~\dots,~a_{n-1}t_n = a_n s_n,~ a_n t_n = b.
\end{equation*}
\end{dfn}

\begin{proposition}
\label{prop:cc2}
Suppose that $M$ \textbf{does not} satisfy the right Ore condition. Then the connected components functor $C$ fails to preserve $\Set$-indexed powers; in particular, $C$ is not cartesian-closed.
\end{proposition}
\begin{proof}
We shall construct an indecomposable right $M$-set $X$ such that $X^{\Delta(\Nbb)}$ is not indecomposable.

Since $M$ does not satisfy the right Ore condition, there is some pair of elements $a,b$ with $aM \cap bM = \emptyset$. Construct the $M$-set $S$ by quotienting by the equivalence relation $m \sim m'$ iff $m = m'$ or $m,m' \in aM$ or $m,m' \in bM$. The quotient map $M \too S$ sends the ideals $aM$ and $bM$ to distinct fixed points of $S$; abusing notation we call these fixed points $a$ and $b$ respectively.

Now let $X$ be the quotient of $\bigsqcup_{n \in \Nbb} S$ by the equivalence relation identifying the element $b$ of the $n$th copy of $S$ with the element $a$ of the $(n+1)$th copy of $S$. Denoting the image of $a$ from the $n$th copy by $a_n$, we observe that for and $k \in \Nbb$, $a_0$ and $a_k \in X$ cannot be connected by a scheme of length less than $k$.

Elements of $X^{\Delta(\Nbb)}$ can be identified with $\Nbb$-indexed tuples of elements of $X$, which we notate as vectors. Take two elements $\vec{x},\vec{y} \in X^{\Delta(\Nbb)}$. If $\vec{x}$ and $\vec{y}$ can be connected by a scheme of length $1$, then this means that there is some $\vec{z} \in X^{\Delta(\Nbb)}$ and elements $s,t \in M$ such that $\vec{x} = \vec{z}s$ and $\vec{y} = \vec{z}t$. In particular, for each index $n$ we have $x_n = z_ns$ and $y_n = z_nt$ so these are connected by a scheme of length $1$ in $X$. Analogously, if $\vec{x}$ and $\vec{y}$ can be connected by a scheme of length $k$, then $x_n$ and $y_n$ can be connected by a scheme of length $k$ in $X$ for all indices $n \in \Nbb$.

Now let $x_n = a_n$ and $y_n = a_0$ for all $n \in \Nbb$. The resulting elements $\vec{x}$ and $\vec{y}$ are in separate components of $X^{\Delta(\Nbb)}$, since for any $k \in \Nbb$, an element $\vec{z}$ connected to $\vec{y}$ by a scheme of length less than $k$ cannot have $k$th component equal to $a_k$. Thus $X^{\Delta(\Nbb)}$ is not indecomposable, as claimed.
\end{proof}

For the case where $M$ satisfies the right Ore condition and $\Setswith{M}$ is power-connected, we have Theorem \ref{thm:labsorb}.\ref{item:powers} below. Finally, we examine the still stronger condition of cartesian-closed-connectedness.

\begin{proposition}
\label{prop:cc1}
Suppose that $M$ satisfies the right Ore condition and that the connected components functor $C$ is cartesian-closed. Then $M$ is trivial.
\end{proposition}
\begin{proof}
Being cartesian-closed requires that $C(M^M) \cong C(M)^{C(M)} \cong 1$, so $M^M$ is indecomposable. In particular, the projection maps $\pi_1,\pi_2:M \times M \rightrightarrows M$ in $M^M = \HOM_M(M \times M, M)$ must be in the same component under the action described above. But since $M$ has the right Ore condition, they are in the same component if and only if there are $m_1,m_2$ with $\pi_1 \cdot m_1 = \pi_2 \cdot m_2$. By inspection of the action, $\pi_2 \cdot m_2 = \pi_2$ for all $m_2 \in M$, while $\pi_1 \cdot m_1$ is independent of the second argument for any $m_1 \in M$, so the same must also be true of $\pi_2$, which forces $M$ to be trivial.
\end{proof}

\subsection{Right Collapsibility and Total Connectedness}

We can express localness of a connected geometric morphism $f$ (Definition \ref{dfn:local}) as the existence of a right adjoint to $f$ in the 2-category of Grothendieck toposes and geometric morphisms; see \cite[Theorem C3.6.1]{Ele}. From this perspective, the dual property to localness, appearing in \cite[Theorem C3.6.14]{Ele}, is total connectedness.

\begin{dfn}
A geometric morphism $f:\Fcal \to \Ecal$ is called \textbf{totally connected} if it is locally connected and the left adjoint $f_!$ preserves finite limits.
\end{dfn}

The following is adapted from Johnstone's results, \cite[Example C3.6.17(a)]{Ele} and \cite[Lemma 1.1]{PLC}. 
\begin{proposition}
\label{prop:veryconnected}
Let $X$ be a sober topological space. Then the following are equivalent:
\begin{enumerate}
\item $\Sh(X)$ is totally connected.
\item $\Sh(X)$ is strongly connected.
\item $\Sh(X)$ is finitely power-connected.
\item Any non-empty open subset of $X$ is connected and a finite intersection of these is nonempty.
\item $X$ is \textbf{irreducible}: if $X = X_1 \cup X_2$ for closed subsets $X_1$ and $X_2$, then $X_1 = X$ or $X_2 = X$.
\item $X$ has a \textbf{dense point}, i.e.\ a point that is contained in all non-empty open sets.
\end{enumerate}
\end{proposition}
\begin{proof}
($1 \Rightarrow 2 \Rightarrow 3$) These implications are trivial.
 
($3 \Rightarrow 4$) Let $X$ be a connected, locally connected topological space such that $\Sh(X)$ is finitely power-connected with global sections geometric morphism $f : \Sh(X) \to \Set$. For two open subsets $U$ and $V$, their product as subterminal objects in $\Sh(X)$ is given by the intersection $U \cap V$. Since $f_!$ preserves finite powers, we know that the natural map $f_!(U) \to f_!(U) \times f_!(U)$ coincides with the diagonal and is an isomorphism, which shows that $U$ is connected whenever it is non-empty (since $U$ has at least one connected component). Moreover, for non-empty $U,V$, if we had $U \cap V$ empty, $U \cup V$ would be a disconnected open set, a contradiction.

($4 \Leftrightarrow 5$) Given closed subsets $X_1,X_2 \subset X$ with $X_1 \cup X_2 = X$, we have $(X - X_1) \cap (X - X_2) = \emptyset$. Given that intersections of non-empty open sets are non-empty, this means that one of $(X - X_1)$ or $(X - X_2)$ must be empty. Hence $X$ is irreducible. Conversely, given two disjoint open subsets, their complements are closed and cover $X$, so one of them must be empty. 

($4 \Rightarrow 6$) Condition ($4$) ensures that the non-empty open sets of $X$ form a completely prime filter, and hence correspond to a point contained in every open set.

($6 \Rightarrow 1$) Having a dense point ensures that $X$ is connected and locally connected. Expressing $\Sh(X)$ as the category of local homeomorphisms over $X$, each connected component of an object $E \to X$ must meet the fibre over the dense point in exactly one point. The inverse image functor of (the geometric morphism corresponding to) this point, which gives the set of points of $E$ lying in the fibre over it, is therefore isomorphic to the connected components functor $f_!$. Being the inverse image functor of a geometric morphism means that $f_!$ preserves finite limits, as required.
\end{proof}

For a ring $R$ without non-zero nilpotent elements, $\mathrm{Spec}(R)$ is irreducible (for the Zariski topology) if and only if $R$ does not have zero divisors (i.e.\ $R$ is a domain).

As in the last section, we can express $C$ as sending an $M$-set to its colimit as a functor. Thus $C$ preserves finite limits if and only if colimits of shape $M\op$ commute with finite limits in $\Set$. A well-known result regarding commuting limits and colimits is that colimits of shape $\Dcal$ commute with finite limits in a topos if and only if $\Dcal$ is \textbf{filtered}, which means concretely that:
\begin{itemize}
	\item $\Dcal$ is non-empty,
	\item For any pair of objects $P,Q$ of $\Dcal$, there is a cospan from $P$ to $Q$.
	\item For any pair of parallel morphisms $f,g:P \rightrightarrows Q$ there is some $h:Q \to R$ in $\Dcal$ with $hf = hg$.
\end{itemize}
We correspondingly say that $\Dcal\op$ is \textbf{cofiltered} if $\Dcal$ satisfies these conditions. Applying this to $M$ as a one-object category as in the last section, we see that the first two conditions are trivial.
\begin{definition}[see {\cite{sedaghatjoo-khaksari}, \cite[III, Definition 14.1]{MAC}}]
We say that $M$ is \textbf{right collapsible} if for any pair $m_1,m_2$ of elements of $M$, there exists $m \in M$ with $m_1m=m_2m$, that is, if $M$ is cofiltered as a category.
\end{definition}

\begin{theorem}[Conditions for $\Setswith{M}$ to be totally connected]
\label{thm:totally}
The following are equivalent:
\begin{enumerate}
	\item $C$ preserves finite limits, i.e.\ $\Setswith{M}$ is totally connected, or the terminal left $M$-set $1$ is flat.
	\item $M$ is cofiltered as a category.
	\item $M$ is right collapsible.
	\item $C$ preserves pullbacks. 
	\item $C$ preserves equalizers. 
	\item The category of points of $\Setswith{M}$ has a terminal object.
\end{enumerate}
\end{theorem}
\begin{proof}
Equivalence of $(1),(2)$ and $(3)$ follows from the discussion above.

($1 \Rightarrow 4 \Rightarrow 5$) The first implication is trivial, the latter is Proposition \ref{prop:flatness-properties}.

($5 \Rightarrow 3$) For $m_1,m_2 \in M$, consider the diagram
\begin{equation*}
\begin{tikzcd}
M \ar[r,shift left=1,"{m_1 \cdot}"] \ar[r,shift right=1,"{m_2 \cdot}"'] & M
\end{tikzcd}.
\end{equation*}
Because $C$ preserves equalizers, the equalizer of this diagram is non-empty, so there is an $m \in M$ with $m_1m = m_2m$.

($1 \Leftrightarrow 6$) The category of points of $\Setswith{M}$ can be identified with the category of flat left $M$-sets and homomorphisms of $M$-sets between them. Since in particular $M$ (with left action given by multiplication) is flat, the category of points has a terminal object if and only if $1$ is flat (on the left), since any other non-empty $M$-set admits more than one homomorphism from $M$.
\end{proof}

\begin{rmk}
Note that a functor into $\Set$ is flat in the sense of Definition \ref{dfn:filtering} if and only if its category of elements is filtered in the above sense; since the general definition of the category of elements is rather involved, we mention only in passing that the category of elements for the terminal left $M$-set is precisely $M\op$, which gives an alternative proof of the equivalence ($1 \Leftrightarrow 2$) in the above. 
\end{rmk}

\begin{remark}
In the above we recovered Sedaghatjoo and Khaksari's result \cite[Lemma 3.7]{sedaghatjoo-khaksari}, which is the equivalence ($1 \Leftrightarrow 3$). The equivalence ($4 \Leftrightarrow 5$) can also be seen as the statement that the left $M$-set with one element is pullback-flat if and only if it is equalizer-flat. In the semigroup literature, it is shown more generally that equalizer-flatness and pullback-flatness coincide for cyclic $M$-sets, see Kilp \textit{et al.} \cite[III,Theorem 16.7]{MAC}.
\end{remark}

In \cite[Corollary 3.8]{sedaghatjoo-khaksari}, Sedaghatjoo and Khaksari show that a monoid $M$ is right collapsible if and only if it is right Ore and $M \times M$ is indecomposable as a right $M$-set with the diagonal action. By the characterisations in preceding sections, this means that $\Setswith{M}$ is totally connected if and only if it is de Morgan and strongly connected. We give an alternative proof of this fact:
\begin{proposition}
\label{prop:totally2}
Let $M$ be a monoid. Then $\Setswith{M}$ is totally connected if and only if it is de Morgan and strongly connected.
\end{proposition}
\begin{proof} %could this be made more streamlined/categorical?
If $C$ preserves finite limits, then it certainly also preserves finite products and monomorphisms. Conversely, suppose that $C$ preserves finite products and monomorphisms. To show that $C$ preserves all finite limits, it is enough to show that $C$ preserves equalizers. Let $E$ be the equaliser of the diagram
\begin{equation*}
\begin{tikzcd}
X \ar[r,shift left=1,"{f}"] \ar[r,shift right=1,"{g}"'] & Y.
\end{tikzcd}
\end{equation*}
Since $C(E)$ is a subobject of $C(X)$ and $C$ preserves coproducts, it is enough to show that $E$ is non-empty whenever $X$ and $Y$ are indecomposable. Because $C$ preserves finite products, $X \times Y$ is indecomposable. Therefore, consider the two (non-empty) subobjects
\begin{equation*}
\{ (x,y) \in X \times Y : f(x) = y \} \qquad \{ (x,y) \in X \times Y : g(x) = y \}.
\end{equation*}
Their intersection is isomorphic to $E$, and is non-empty by part \ref{item:deMorgan-7} of Theorem \ref{thm:deMorgan}.
\end{proof}

Proposition \ref{prop:totally2} enables us to show that when $M$ is a monoid with $\Setswith{M}$ strongly connected, $C$ need not preserve all monomorphisms. This demonstrates (for example) that the properties of a topos of being de Morgan and strongly connected are independent.
\begin{example}[Strongly connected $\not\Rightarrow$ de Morgan]
Let $M = \{1,a,b\}$ be the three-element monoid with $a$ and $b$ right-absorbing. In this case, $\PSh(M)$ is known as the topos of reflexive graphs. This topos also appears in the work of Connes and Consani \cite{connes-consani-gromov}, where the objects of this topos are seen as sets equipped with a certain notion of reflexive relation. It follows from Lemma \ref{lem:localstrong} that $\PSh(M)$ is strongly connected. However, $aM \cap bM = \varnothing$ so $\PSh(M)$ is not de Morgan.
\end{example}

\subsection{Left absorbing elements and Colocalness}

\begin{dfn}
\label{dfn:colocal}
A more direct dual of Definition \ref{dfn:local} in the context of essential geometric morphisms is the existence of an `extra left adjoint'. We say that a locally connected Grothendieck topos $\Ecal$ with global sections morphism $f$ is \textbf{colocal} if the left adjoint $f_!$ has a further left adjoint. More generally, we might say a locally connected geometric morphism $f : \Fcal \to \Ecal$ is \textbf{colocal} if its left adjoint $f_!$ has a further $\Ecal$-indexed left adjoint.
\end{dfn}

There is a characterisation of topological spaces $X$ such that $\Sh(X)$ is colocal comparable to Proposition \ref{prop:localspace}.
\begin{proposition}
\label{prop:colocalspace}
Let $X$ be a locally connected sober topological space. Then $\Sh(X)$ is a colocal topos if and only if $X$ has a (necessarily unique) \textbf{dense open point}, i.e.\ a point $x_0 \in X$ such that $\{x_0\}$ is an open set that is contained in all other open sets.
\end{proposition}
\begin{proof}
Suppose that $X$ has a dense open point $x_0$. From the proof of Proposition \ref{prop:veryconnected}, we know that the connected components functor $f_!$ of $\Sh(X)$ coincides with the inverse image functor for the geometric morphism corresponding to $x_0$. Moreover, we can construct a left adjoint to $f_!$ which maps a set $S$ to the sheaf $F_S$ defined as
\begin{equation*}
F_S(U) = \begin{cases}
S \quad & \text{if }U = \{x_0\} \\
\varnothing \quad & \text{otherwise.}
\end{cases},
\end{equation*}
So $\Sh(X)$ is a colocal topos.

Conversely, suppose that $\Sh(X)$ is a colocal topos. Then the connected components functor preserves arbitrary limits. In particular, it preserves the terminal object and monomorphisms, which shows that each non-empty open subset of $X$ is connected. Now consider the diagram $\{U_i\}_{i \in I}$ of all non-empty open subsets of $X$. We have $C(\bigwedge_{U \in \Ocal(X)} U) = \bigwedge_{U \in \Ocal(X)} C(U) = 1$, so $\bigwedge_{U \in \Ocal(X)} U$ is a minimal non-empty open subset. Because $X$ is sober, this minimal open subset contains exactly one point.
\end{proof}

For a commutative ring $R$ without zero-divisors, we find that the topos of sheaves on $\mathrm{Spec}(R)$ (with the Zariski topology) is colocal if and only if there is an $f \in R$ such that $R[f^{-1}]$ is a field (necessarily the field of fractions of $R$). In this case, $R$ is called a Goldman domain. If we assume that $R$ is noetherian, then $R$ is a Goldman domain if and only if $R$ has only finitely many prime ideals; see \cite[Theorem 12.4]{PeteLClarkNotes}.

\begin{lemma}
\label{lem:lconnected}
Suppose $M$ is a monoid with a left absorbing element $l$. Then for any right $M$-set $A$, we have $C(A) \cong Al$.
\end{lemma}
\begin{proof}
Recall that we can express $C(A)$ as the set of equivalence classes of $A$ under the equivalence relation generated by $a \sim a\cdot m$ for $a \in A$, $m \in M$. Clearly every equivalence class has a representative of the form $a \cdot l$, so it suffices to show that if $a \sim b$ then $a \cdot l = b \cdot l$ (so this representative is unique). Indeed, for $a \sim b$ to hold there must be a finite sequence of elements of $A$, $a=a_0,a_1,\dotsc,a_k=b$ and elements $m_0,\dotsc,m_{k-1}$ and $n_1,\dotsc,n_k$ with $a_i \cdot m_i = a_{i+1} \cdot n_{i+1}$ for $i=0,\dotsc,k-1$. Then we have $a_i \cdot l = a_i \cdot m_il = a_{i+1} \cdot n_{i+1}l = a_{i+1} \cdot l$ and so inductively $a \cdot l = b \cdot l$, as required.
\end{proof}

The above lemma is the dual of Lemma \ref{lem:rfixed}. We will use it to prove $(1 \Rightarrow 6)$ in the following theorem.

\begin{thm}[Conditions for $\Setswith{M}$ to be colocal]
\label{thm:labsorb}
Let $M$ be a monoid, $C:\Setswith{M} \to \Set$ its connected components functor and $\Gamma:\Setswith{M} \to \Set$ its global sections functor. The following are equivalent:
\begin{enumerate}
	\item $M$ has a left absorbing element.
	\item $M$ has a minimal non-empty right ideal whose monoid of endomorphisms as a right $M$-set is trivial. \label{item:minimal}
	\item The category of indecomposable right $M$-sets has an initial object.
	\item Every indecomposable right $M$-set has a minimal non-empty subobject admitting no non-trivial endomorphisms. \label{item:minimal2}
	\item The category of essential points of $M$ has a terminal object.
	\item $C$ has a left adjoint $c$ which preserves connected limits.
	\item $C$ has a left adjoint $c$; that is, $\Setswith{M}$ is colocal over $\Set$.
	\item $C$ preserves arbitrary products, so the terminal left $M$-set is product-flat.
	\item $C$ preserves $\Set$-indexed powers. \label{item:powers}
\end{enumerate}
\end{thm}
\begin{proof}
($1 \Rightarrow 2$) Any non-empty right ideal of $M$ must contain the collection of left-absorbing elements, and the subset of left-absorbing elements is a right ideal of $M$. Moreover, any $M$-endomorphism of an ideal of $M$ must fix the set of left-absorbing elements (since $f(l) = f(l) \cdot l = l$ for any left-absorbing $l$), so this minimal ideal has no non-trivial endomorphisms.

($2 \Rightarrow 3$) Let $l$ be a left-absorbing element. Then the set of left-absorbing elements can be written as $lM$, in particular it is indecomposable projective as a right $M$-set. We claim that $lM$ is an initial object in the category of indecomposable right $M$-sets. Take an indecomposable right $M$-set $X$. Note that morphisms $lM \to X$ correspond to elements of $Xl \cong X \otimes_M Ml$. Because $Ml = 1$ as a left $M$-set, we have $Xl = C(X)=1$, so there is a unique morphism $lM \to X$.

($3 \Rightarrow 4$) Let $A$ be the initial object in the category of indecomposable right $M$-sets. Then for every indecomposable right $M$-set $X$, there is a unique morphism $f: A \to X$. The image $Q$ of $f$ is necessarily contained in every indecomposable subobject of $X$, which in turn implies that it is contained in every subobject. Further, the unique morphism $\pi : A \to Q$ is an epimorphism, so for any endomorphism $g : Q \to Q$ we have $g \circ \pi = \pi$, which shows $g = 1$. 

($4 \Rightarrow 1$) Consider the minimal non-empty subobject $A$ of $M$. For arbitrary $m \in M$, take the morphism $f: A \to M$, $a \mapsto ma$. Then $f(A)$ contains $A$ and $f^{-1}(A) = A$ by minimality of $A$. So $f$ defines an endomorphism of $A$, which is trivial by assumption. This shows that $ma=a$, for all $m \in M$ and all $a \in A$. In other words, each element of $A$ is a left-absorbing element. 

($1 \Leftrightarrow 5$) Recall that the category of essential points can be identified with the category of indecomposable projective left $M$-sets. If $l$ is a left absorbing element, then the terminal left $M$-set is projective, since we can write it as $1 = Ml$. Conversely, if the category of essential points has a terminal object, then $1$ is projective, so there is an idempotent $e \in M$ such that $1 = Me$. But then $e$ is a left absorbing element.

($1 \Rightarrow 6$) By Lemma \ref{lem:lconnected}, if $l$ is a left-absorbing element of $M$ then $C = \HOM_M(lM,-)$, so $C$ preserves limits. It follows from the Special Adjoint Functor Theorem that $C$ has a left adjoint $c$. Using Proposition \ref{prop:adjunction}, we know that $c(X) \cong X \times lM$, whence it preserves connected limits. Note that it preserves products if and only if $lM$ has a single element (so $l$ is a zero element).

($6 \Rightarrow 7 \Rightarrow 8 \Rightarrow 9$) These implications are trivial.

($9 \Rightarrow 1$) By Proposition \ref{prop:cc2}, if $\Setswith{M}$ is power-connected, then $M$ must satisfy the right Ore condition. Consider the product of $|M|$ copies of $M$ in $\Setswith{M}$. This is indecomposable by assumption, so there are elements $s,t \in M$ such that $(m)_{m \in M} \cdot s = (1)_{m \in M} \cdot t$, which is to say such that $ms = t$ for all $m \in M$. Taking $m=1$ we have that $s=t$ is a left-absorbing element, as required.
\end{proof}

The equivalence of conditions 1 and 7 appears as Proposition 3.9 of \cite{sedaghatjoo-khaksari}; we underline once again that their `right zero' elements are our `left absorbing' elements.

\begin{remark}
Between the preservation of finite products by $C$ in Proposition \ref{prop:strongcon} and the preservation of arbitrary products in Theorem \ref{thm:labsorb}, we can also investigate intermediate sizes of products, which can be equivalently stated as the requirement that colimits over $M\op$ commute with such products (see the discussion in Section \ref{ssec:strong-connectedness}). Surprisingly, by \cite[Theorem 3.1]{adamek-boubek-velebil} for an arbitrary small category $\Dcal$, commuting with even countably infinite products forces commutation with equalizers and hence with all equally large limits. That is, we may as well expand our considerations from products to arbitrary limits. We can also conclude that the equivalence ($7 \Leftrightarrow 8$) in Theorem \ref{thm:labsorb} is true in general for presheaf toposes.

We recall some classical terminology. Let $\kappa$ be a regular cardinal. Then a \textbf{$\kappa$-small category} is a category for which the collection of morphisms is a set of cardinality \textit{strictly smaller} than $\kappa$. Further, \textbf{$\kappa$-small limits} are limits of diagrams indexed by $\kappa$-small categories. For example, and $\omega$-small limits are finite limits, for $\omega = |\Nbb|$. We say that a category $\Dcal$ is \textbf{$\kappa$-filtered} if every diagram $F : I \to \Dcal$, with $I$ a $\kappa$-small category, has a cone over it; this is equivalent to $\Dcal$-colimits commuting with $\kappa$-small limits. Dually, $\Dcal$ is \textbf{$\kappa$-cofiltered} if $\Dcal\op$ is $\kappa$-filtered.

It follows that $C$ preserves $\kappa$-small limits if and only if $M$ has the property that for any family $\{m_i\}_{i \in I}$ with $|I|<\kappa$, there is an $m$ such that:
\begin{equation*}
m_i m = m_j m
\end{equation*}
for all $i,j \in I$. We may call monoids with this property \textbf{right $\kappa$-collapsible}.

Suppose that $M$ is right $\kappa$-collapsible for some $\kappa > |M|$. Then there is some $l \in M$ such that $ml = m'l$ for all $m,m' \in M$, so by taking $m' = 1$ we see that $l$ is left absorbing (see also the proof of Theorem \ref{thm:labsorb}). However, it is possible that $M$ is $|M|$-collapsible but does not have a left-absorbing element. We can construct examples as follows. Let $\kappa$ be a regular cardinal. Then we can identify $\kappa$ with the set of ordinals of cardinality strictly smaller than $\kappa$. The union of two ordinals in $\kappa$ is still in $\kappa$, so the union defines a (commutative idempotent) monoid structure on $\kappa$. Now let $\{\alpha_i\}_{i \in I}$ be a family of ordinals in $\kappa$, with $|I|<\kappa$. Then the union $\alpha = \bigcup_{i \in I} \alpha_i$ is again in $\kappa$, because $\kappa$ is regular. Clearly, $\alpha_i \cup \alpha = \alpha \cup \alpha_j$ for all $i,j \in I$. So this monoid is right $\kappa$-collapsible, but it does not have a left absorbing element.
\end{remark}

\subsection{Zero Elements}

\begin{dfn}
Let $f:\Fcal \to \Ecal$ be a connected, locally connected geometric morphism. We say $f$ is \textbf{bilocal} if it is both local and colocal. We say $f$ is \textbf{bipunctually locally connected} if it is both punctually and copunctually locally connected. As usual, we say a Grothendieck topos has these properties if its global sections morphism does.
\end{dfn}

In \cite[Definitions 1 and 2]{cohesion}, Lawvere introduced the terms \textbf{of quality type} and \textbf{category of cohesion} over a base category. For Grothendieck toposes over $\Set$, these respectively coincide with the bipunctual local connectedness presented above and condition \ref{item:cohesion} of Theorem \ref{thm:zero} below, so that in the case of toposes of the form $\Setswith{M}$ they coincide.

\begin{thm}[Conditions for $\Setswith{M}$ to be bilocal]
\label{thm:zero}
Let $M$ be a monoid, $\Gamma$ and $C$ the usual functors and $\alpha:\Gamma \to C$ the natural transformation of Lemma \ref{lem:alpha}. The following are equivalent:
\begin{enumerate}
	\item $M$ has a zero element.
	\item $\Gamma$ is full.
	\item $\alpha$ is a split monomorphism.
	\item $\alpha$ is an isomorphism: $\Setswith{M}$ is bilocally punctually connected.
	\item $\Gamma$ has a right adjoint and $C$ has a left adjoint: $\Setswith{M}$ is bilocal.
	\item $\alpha$ is epic and $C$ preserves $\Set$-indexed powers. \label{item:cohesion}
	\item $\Gamma$ has a right adjoint and $C$ preserves monomorphisms.
	\item $C$ has a left adjoint which preserves the terminal object of $\Set$.
\end{enumerate}
\end{thm}
\begin{proof}
($1 \Rightarrow 2$) Each component of each $M$-set has a unique fixed point, obtained by acting by the unique zero element. Given a morphism $g: \Gamma(X) \to \Gamma(Y)$ we can map each element of $X$ to the fixed point in the same component and then apply $g$ to obtain an $M$-set homomorphism $X \to Y$ whose image under $\Gamma$ is $g$.

($2\Rightarrow 3 \Rightarrow 4$) Let $\epsilon$ be the counit of $(\Delta \dashv \Gamma)$. Then $\Gamma$ is full if and only if $\epsilon$ is a split monomorphism, which makes $C\epsilon$ and hence $\alpha$ a split monomorphism. In particular, $\alpha_M: \Gamma(M) \to C(M) = 1$ is a split monomorphism, meaning there is a morphism $1 \to \Gamma(M)$ in $\Set$ and so $\Gamma(M)$ is non-empty, which means $M$ has a right absorbing element and $\alpha$ must be epic and hence an isomorphism.

($4 \Rightarrow 5$) Since $\Gamma$ and $C$ are naturally isomorphic, $\Delta$ is a right and left adjoint to both of them.

($5 \Leftrightarrow 6$) By Theorem \ref{thm:rabsorb}, we have that $\alpha$ is epic if and only if $\Gamma$ has a right adjoint and by Theorem \ref{thm:labsorb} we have that $C$ has a left adjoint if and only if $C$ preserves $\Set$-indexed powers.

($5 \Rightarrow 7$) This is immediate.

($7 \Rightarrow 1$) By Lemma \ref{lem:localstrong}, $M$ being local makes $\Setswith{M}$ strongly connected, and so $C$ also preserving monomorphisms makes $\Setswith{M}$ totally connected and hence $M$ is right collapsible by Proposition \ref{prop:totally2} and Theorem \ref{thm:totally}. Applying right collapsibility to any pair of right absorbing elements shows that they must be equal, so there is a unique right absorbing element, which is thus a zero element.

($5 \Leftrightarrow 8$) One direction follows from the fact that $\Delta$ preserves $1$. Conversely, any left adjoint functor whose domain is $\Set$ is determined by the image of $1$. In particular, if the left adjoint of $C$ preserves $1$ it is naturally isomorphic to $\Delta$, and hence their right adjoints are naturally isomorphic too, so $\Delta$ is also a right adjoint to $\Gamma$.
\end{proof}

Note that ($4 \Leftrightarrow 7$) appears in a more general form as \cite[Proposition 3.7]{PLC}. Observe also that $\Gamma$ being full implies that $C$ is full, but now in an entirely constructive way!

\begin{remark}
Since any monoid with a right-absorbing element has either exactly one, which is necessarily a zero element, or at least two, we have a dichotomy between Lawvere's sufficiently cohesive toposes from Remark \ref{rmk:sufficiently} above and their toposes of quality type. This dichotomy is shown by Menni in \cite[Corollary 4.6]{menni-continuous-cohesion} to hold for arbitrary pre-cohesive presheaf toposes.
\end{remark}

\subsection{Trivialising Conditions}
\label{ssec:trivial}

Many of the conditions encountered in this article suggest lines of investigation for further properties. As it turns out, many of these directions turn out to be dead ends, in the sense that they force the monoid to be trivial. In this section we present a variety of these conditions. As promised in an earlier section, we include some alternative weakenings of the concept of cartesian-closedness in this list.

\begin{dfn}
Suppose $F:\Fcal \to \Ecal$ is a functor between toposes which preserves products. $F$ is \textbf{sub-cartesian-closed} if the comparison morphisms $\theta_{P,Q}$ of \eqref{eq:theta} in Section \ref{ssec:toposes} are monomorphisms for every pair $P,Q$ of objects of $\Fcal$. If $F$ moreover preserves monomorphisms and the comparison morphism $\chi$ of \eqref{eq:chi} is a monomorphism, we say $F$ is \textbf{sublogical}.
\end{dfn}

Sublogical functors appear in the definition of \textbf{open} geometric morphisms: a geometric morphism is called open if its inverse image functor is sub-cartesian-closed. We therefore refer the reader once again to Johnstone \cite[Section C3.1]{Ele} for background on this concept, where a different but equivalent definition is given. Since any hyperconnected morphism is open, \cite[Corollary C3.1.9]{Ele}, we have that $\Delta$ is sub-logical for any monoid $M$.

\begin{thm}[Conditions for $\Setswith{M}$ to be equivalent to $\Set$]
\label{thm:trivial}
Let $M$ be a monoid, $\Setswith{M}$ its topos of right actions, and $\Gamma, \Delta, C$ the usual functors. The following are equivalent:
\begin{enumerate}
	\item $M$ is the trivial monoid.
	\item The geometric morphism $(\Delta \dashv \Gamma)$ is an equivalence.
	\item $\Gamma$ is full and faithful, or the above geometric morphism is an inclusion of toposes.
	\item $\Gamma$ is faithful.
	\item The geometric morphism $(\Delta \dashv \Gamma)$ is localic.
	\item $C$ is full and faithful.
	\item $C$ is faithful.
	\item $\Gamma$ is cartesian-closed or logical.
	\item $\Gamma$ is sub-cartesian-closed or sub-logical.
	\item $\Delta$ is logical and $C$ preserves products.
	\item $C$ is logical, cartesian-closed, sub-logical or sub-cartesian-closed. \label{item:trivial10}
	\item $\Gamma$ reflects binary coproducts or binary products or the terminal object or monomorphisms.
\end{enumerate}
\end{thm}
\begin{proof}
($1 \Leftrightarrow 2$) This is immediate after noting that the trivial monoid is the only monoid which can represent $\Set$ as a presheaf topos.

($2 \Leftrightarrow 3 \Rightarrow 4 \Rightarrow 5 \Rightarrow 2$) Since any equivalence is an inclusion and $\Gamma$ is faithful if and only if the counit of $(\Delta \dashv \Gamma)$ is epic, which is sufficient to make the geometric morphism localic. But a geometric morphism which is hyperconnected and localic is automatically an equivalence.

($2 \Rightarrow 6 \Rightarrow 7 \Rightarrow 5$) The components of an equivalence are always full and faithful. The second implication is trivial, and $C$ is faithful if and only if the unit of $(C \dashv \Delta)$ is a monomorphism, which is again sufficient to make the geometric morphism localic.

($3 \Leftrightarrow 8 \Rightarrow 9 \Leftrightarrow 3$) Since $\Gamma: \Setswith{M} \to \Set$ always preserves the subobject classifier by \cite[Proposition A4.6.6(v)]{Ele}, it is logical if and only if it is cartesian-closed, and the latter is equivalent to the global sections morphism being full and faithful by Lemma A4.2.9 there. Being sub-logical (or equivalently sub-cartesian-closed) is an apparently weaker condition, but is still equivalent to the geometric morphism being an inclusion by \cite[Proposition C3.1.8]{Ele}.

($1 \Leftrightarrow 10$) This is the content of Example \ref{xmpl:groupnotstrong}.

($2 \Rightarrow 11 \Rightarrow 1$) The components of an equivalence are always logical. Conversely, we observed in the proof of Theorem \ref{thm:deMorgan} that $C(\Omega)$ always has one or two elements, so that the comparison morphism $\chi$ for $C$ is always monic. Being sublogical is therefore equivalent to being sub-cartesian-closed. The remaining implication is contained in the proofs of Propositions \ref{prop:cc2} and \ref{prop:cc1}, since in both cases we actually showed that one of the comparison morphisms failed to be monic.

($2 \Rightarrow 12 \Rightarrow 1$) All of the functors involved in an equivalence preserve and reflect all limits and colimits. If $\Gamma$ reflects coproducts, consider two cases. First, if $\Gamma(M) = \emptyset$, then $\Gamma(1 \sqcup M) = \Gamma(1) \sqcup \Gamma(\varnothing)$ forces $1\sqcup M \cong 1\sqcup \varnothing$, a contradiction. On the other hand, if $\Gamma(M)$ is non-empty then $M$ has some right absorbing elements; consider the $M$-set $X$ obtained by identifying them all. Then in particular $\Gamma(X) = 1 = \Gamma(1)$ and hence $\Gamma(X\sqcup 1) \cong \Gamma(1) \sqcup  \Gamma(1)$ forces $X\sqcup 1 \cong 1\sqcup 1$ and hence $X \cong 1$; it follows that every element of $M$ was right-absorbing and so $M \cong 1$ as required. The arguments for binary products, the terminal object or monomorphisms are similar: replace the reflected coproducts with the reflections of $\Gamma(M \times 1) = \Gamma(\varnothing) \times \Gamma(1)$ or $\Gamma(X \times 1) = \Gamma(1) \times \Gamma(1)$ in the first case; $\Gamma(M\sqcup 1) = \Gamma(1)$ and $\Gamma(X) = \Gamma(1)$ in the second case and with the morphisms $M \too 1$ and $X \too 1$ in the final case.
\end{proof}

\section{Conclusion}
\label{sec:conclusion}

We can summarise some of the properties and results obtained in this paper in Table \ref{table:results1} and Table \ref{table:results2}.
\begin{table}[h!] 
\caption{Summary of results regarding $\Gamma$} \label{table:results1}
\begin{tabularx}{\linewidth}{Z|ZZZZ} 
\toprule[0.1em]
Topos property     & Topological property        & Monoid property                   & $\Gamma$ preserves                                          & Examples                                                                   \\ \midrule[0.1em]
Local              & $\exists$ focal point       & $\exists$ right \mbox{absorbing} \mbox{element} & all colimits (equiv. \mbox{epimorphisms})  & multiplicative monoid of a ring, $\End(S)$                        \\ \hline
Strongly compact   & e.g.\ spectral or compact Hausdorff              & right-factorably finitely generated   & filtered colimits                               & as in the cell above, as well as finitely generated monoids                \\ 
\bottomrule[0.1em]
\end{tabularx}
\end{table}

\begin{table}[h!] 
\caption{Summary of results regarding $C$} \label{table:results2}
\begin{tabularx}{\linewidth}{Z|ZZZZ} 
\toprule[0.1em]
Topos property     & Topological property        & Monoid property    & $C$ preserves                         & Examples                                                                   \\ \midrule[0.1em]
Colocal            & $\exists$ open dense point  & $\exists$ left absorbing element                                   & all limits (equiv. products)          & $\End(S)\op$                                                      \\ \hline
Totally connected  & irreducible                 & right collapsible                                     & finite limits (equiv. equalizers)     & $\End(S)\op$, $(\mathbb{N},\mathrm{max})$                         \\ \hline
Strongly connected & irreducible                 & $M \times M$ is indecomposable                            & finite products                       & $\End(S)\op$, $\End(S)$                                  \\ \hline
de Morgan          & extremally disconnected     & right Ore                                                & monomorphisms                         & $\End(S)\op$, any commutative monoid \\ \bottomrule[0.1em]
\end{tabularx}
\end{table}

\subsection{Notable Omissions}

This article is far from an exhaustive presentation of what topos-theoretic properties mean for toposes of the form $\Setswith{M}$ and the monoids presenting them: we have chosen to focus on those properties expressible in terms of the functors constituting the global sections geometric morphism. Other properties which the authors are already investigating include:
\begin{itemize}
\item \textit{Classifying topos properties}: When is $\Setswith{M}$ the classifying topos of a regular or coherent theory? More generally, what properties of theories classified by $\Setswith{M}$ and their categories of models can be deduced from properties of $M$ and vice versa?
\item \textit{Diagonal properties}: Since the (2-)category of toposes and geometric morphisms has pullbacks, any geometric morphism $\Fcal \to \Ecal$ induces a diagonal $\Fcal \to \Fcal \times_{\Ecal} \Fcal$. We may in particular apply this to the global sections morphism, to express properties such as \textit{separatedness} of a topos (cf. \cite[Definition C3.2.12(b)]{Ele}). However, a more detailed understanding of geometric morphisms between toposes of the form $\Setswith{M}$ is needed to analyse these.
\item \textit{Relative properties}: Some properties of (Grothendieck) toposes are most succinctly expressed by the existence of geometric morphisms of a particular type to or from toposes with certain properties (see below for an example).
\item \textit{Categorical properties}: There are some categorical properties of Grothendieck toposes that ostensibly aren't expressible in terms of the global sections morphism in a straightforward way, although they might be expressible in the relative sense above. These include every object having a certain property, or there being `enough' (a separating set of) objects with a particular property.
\end{itemize}

Each of these classes merits a systematic study in its own right. We briefly mention illustrative examples of the latter two classes of properties, corresponding to more elementary and well-studied properties of monoids.

We say that a Grothendieck topos $\Ecal$ is an \textbf{\'etendue} if there is an object $X$ of $\Ecal$ such that the slice topos $\Ecal/X$ is localic, i.e.\ such that $\Ecal/X$ is equivalent to the category of sheaves on a locale. This can alternatively be stated as the existence of an atomic geometric morphism to $\Ecal$ from a localic topos. By \cite[Lemma C5.2.4]{Ele}, a presheaf topos $\PSh(\Ccal)$ is an \'etendue if and only if every morphism in the category $\Ccal$ is a monomorphism. For a monoid $M$, it follows that $\Setswith{M}$ is an \'etendue if and only if for all $a,b,m \in M$, the equality $ma=mb$ implies that $a = b$. Monoids with this property are usually called \textbf{left cancellative}. 

An object $A$ of a topos is called \textbf{decidable} if the diagonal subobject $A \hookrightarrow A \times A$ has a complement. In particular, if $A$ is a right $M$-set, then $A$ is decidable if for two distinct elements $a,b \in A$ we have $a \cdot m \neq b \cdot m$ for all $m \in M$. Subobjects of decidable objects are again decidable. We say a topos is \textbf{locally decidable} if every object is a quotient of a decidable object. By \cite[Remark C5.4.3]{Ele}, $\Setswith{M}$ is locally decidable if and only if $\Setswith{M\op}$ is an \'{e}tendue. So $\Setswith{M}$ is locally decidable if and only if for all $a,b,m \in M$, the equality $am=bm$ implies that $a = b$, which is to say that $M$ is \textbf{right cancellative}, dually to the above.

In the other direction, there are some notable elementary properties of monoids which we have not yet found a topos-theoretic equivalent for. The most basic is the left Ore condition, dual to Definition \ref{dfn:rOre}; of course, we could simply examine the category of \textit{left} actions of our monoid, and dualize the results presented in this paper, but believe it will be more informative to seek a condition intrinsic to the topos of right actions, given the variety of equivalent conditions we reached in Theorem \ref{thm:deMorgan}.

The internal logic of a topos is embodied in the structure of its subobject classifier. As we have seen, in $\Setswith{M}$ this is determined by the structure of the right ideals of $M$, but we have only tackled the most basic cases in which $\Setswith{M}$ is a Boolean or de Morgan topos. There are therefore a wide array of algebraic or logical properties of the lattice of right ideals that demand further investigation.

\subsection{Relativisation and Generalisation}

In \cite{TDMA}, an equivalence was demonstrated between the 2-category of monoids, semigroup homomorphisms and `conjugations' and the 2-category of their presheaf toposes, essential geometric morphisms between these and geometric transformations. This means that we can just as systematically explore how properties of semigroup or monoid homomorphisms are reflected as properties of essential geometric morphisms between toposes of the form $\Setswith{M}$. This is a direct extension of the work we have done in this article, since the unique homomorphism $M \to 1$ corresponds under this equivalence to the global sections morphism of $\Setswith{M}$. More generally, we will be able to use the tensor-hom expressions described in Section \ref{sec:bg} for adjunctions between these toposes to give an algebraic interpretation of not-necessarily-essential geometric morphisms, as we did in Corollary \ref{crly:tidiness}. Since such tensor-hom expressions exist for geometric morphisms between presheaf toposes more generally (see \cite[Section VII.2]{MLM}), toposes of monoid actions may provide a good context from which to build an algebraic analysis of geometric morphisms.

One direction that the authors are yet to work on is relativisation. Amongst internal categories in arbitrary toposes, monoids are naturally defined as those whose object of objects is the terminal object. Accordingly, one might be interested in examining toposes of internal right actions of internal monoids relative to a topos other than $\Set$. While many of the results we obtain in this article were arrived at constructively or are expressed in a way that relativises directly, there are some which cannot be transferred directly into an arbitrary topos. For instance, our inductive construction of the submonoid right-factorably generated by $S$ in Lemma \ref{lem:construct} requires the presence of a natural number object, while the application of Proposition \ref{prop:cc2} in Theorem \ref{thm:labsorb} relies on the law of excluded middle. More significantly, the proof that condition \ref{item:rab6} implies condition \ref{item:rab7} in Theorem \ref{thm:rabsorb} explicitly relies on a form of the axiom of choice. This investigation will therefore be non-trivial, and it will be interesting to discover the relative analogues of the results presented here.

\bibliographystyle{amsplainarxiv}
\bibliography{monoidprops}

\end{document}